\documentclass[a4paper,11pt]{article}
\usepackage[utf8]{inputenc}
\usepackage{amsmath}
\usepackage{amsfonts}
\usepackage{amssymb}
\usepackage{amsthm}
\usepackage[english]{babel}
\usepackage{fontenc}
\usepackage{graphicx}
\usepackage{subfigure}
\usepackage{mathtools}
\usepackage{accents}
\usepackage{microtype}
\usepackage{dsfont}
\usepackage{enumitem}
\usepackage[normalem]{ulem}
\mathtoolsset{showonlyrefs}
\usepackage[hmargin=2.5cm,vmargin=2.5cm,paper=a4paper]{geometry}
\newcommand{\R}{\mathbb R}
\newcommand{\ls}{\leqslant}
\newcommand{\gs}{\geqslant}
\renewcommand{\div}{\operatorname{div}}
\newcommand{\dist}{\operatorname{dist}}
\newcommand{\per}{\operatorname{Per}}
\newcommand{\sign}{\operatorname{sign}}
\newcommand{\supp}{\operatorname{Supp}}
\newcommand{\conv}{\operatorname{Conv}}
\newcommand{\BV}{\operatorname{BV}}
\newcommand{\TV}{\operatorname{TV}}

\DeclareMathOperator*{\argmin}{arg\,min}

\newcommand{\eps}{\varepsilon}

\newcommand{\loc}{\mathrm{loc}}
\newcommand{\dd}{\, \mathrm{d}}
\newcommand{\Br}{B(x_0,r)}

\newcommand{\eas}{E_\alpha^s}
\newcommand{\wksto}{\stackrel{\ast}{\rightharpoonup}}

\newcommand\restr[2]{{
\left.\kern-\nulldelimiterspace #1 \vphantom{\big|} \right|_{#2} 
}}
\renewcommand{\dh}{d_{\mathcal H}}
\newcommand{\mres}{\mathbin{\vrule height 1.6ex depth 0pt width
0.13ex\vrule height 0.13ex depth 0pt width 1.3ex}}
\DeclareRobustCommand{\rchi}{{\mathpalette\irchi\relax}}
\newcommand{\irchi}[2]{\raisebox{\depth}{$#1\chi$}}
\numberwithin{equation}{section}
\newtheorem{theorem}{Theorem}[section]
\newtheorem{proposition}[theorem]{Proposition}
\newtheorem{corollary}[theorem]{Corollary}
\newtheorem{lemma}[theorem]{Lemma}
\theoremstyle{definition}
\newtheorem{definition}[theorem]{Definition}
\newtheorem{remark}[theorem]{Remark}
\newtheorem{example}[theorem]{Example}
\newtheorem{assumption}[theorem]{Assumption}
\usepackage{color}
\usepackage[dvipsnames]{xcolor}
\usepackage{comment}

\begin{document}
\title{Convergence of level sets in total variation denoising through variational curvatures in unbounded domains}
\author{Jos\'e A. Iglesias\thanks{Johann Radon Institute for Computational and Applied Mathematics (RICAM), Austrian Academy of Sciences, Linz, Austria (\texttt{jose.iglesias{@}ricam.oeaw.ac.at})} , Gwenael Mercier\thanks{Faculty of Mathematics, University of Vienna, Austria (\texttt{gwenael.mercier{@}univie.ac.at})}}
\date{}
\maketitle

\begin{abstract}
We present some results of geometric convergence of level sets for solutions of total variation denoising as the regularization parameter tends to zero. The common feature among them is that they make use of explicit constructions of variational mean curvatures for general sets of finite perimeter. Consequently, no additional regularity of the level sets of the ideal data is assumed, and in particular the subgradient of the total variation at it could be empty. In exchange, other restrictions on the data or on the noise are required. We consider two cases: characteristic functions with a parameter choice depending on the noise level, and noiseless generic data.
\end{abstract}

\section{Introduction and main results}
We aim to provide a precise analysis of the generalized Rudin-Osher-Fatemi denoising scheme based on total variation minimization in the low noise regime, in general dimension and with no source condition assumptions. More precisely, given a real function $\psi:\R\to \R$, some ideal data to be recovered $f:\R^d \to \R$ with compact support, an additive perturbation $w$, as well as a regularization parameter $\alpha >0$, we consider minimizers of 
\begin{equation}
\label{eq:ROFpsi}
 \inf_{u \in \BV(\R^d)} \int_{\R^d} \psi(u-f-w) + \alpha \TV(u).
\end{equation}
We make the following assumptions on the function $\psi$ appearing in the data term and its Fenchel conjugate $\psi^\ast$:
\begin{equation}
\label{eq:psiassum}\begin{gathered} \psi \text{ is a strictly convex and even function, with}\ \psi(0)=0,\\
\psi^\ast \text{ is uniformly convex and }\left|\psi(s)\right|\ls C|s|^{d/(d-1)} \text{ for some }C>0.\tag{A}\end{gathered}
\end{equation}
If $1 < p \ls 2$ the functions $t \mapsto |t|^p/p$ satisfy these convexity properties \cite[Example 5.3.10]{BorVan10}, so the case $p=d/(d-1)$ satisfies all the conditions of Assumption \eqref{eq:psiassum}.
\begin{remark}Since $\psi$ is strictly convex, even and vanishes at zero, we have that $\psi(t) > 0$ for $t\neq 0$. Moreover, $\psi^\ast$ being uniformly convex implies that $\psi$ is differentiable with $\psi'$ uniformly continuous \cite[Thm.~5.3.17, Prop.~4.2.14]{BorVan10}, in particular $\psi \in\mathcal{C}^1(\R)$. Moreover, strict convexity of $\psi$ implies that $\psi^\ast$ is also differentiable \cite[Thm.~5.3.7]{BorVan10}. We will use both of these properties in the sequel.
\end{remark}

We study the regime in which $\alpha$ and $w$ tend to zero simultaneously, for which under natural assumptions it is easy to prove (see Proposition \ref{prop:existconv} below) that the unique minimizers $u_{\alpha, w}$ of \eqref{eq:ROFpsi} converge to $f$ in the strong $L^1_\loc$ topology. In particular, if along a sequence the solutions $u_{\alpha_n,w_n}$ have a common compact support, we have, using Fubini's theorem \cite[Thm.~2]{IglMerSch18}, for a.e.~$s>0$ that
\begin{equation}\label{eq:convlsmeas}\big|\{u_{\alpha_n,w_n} > s\} \Delta \{f > s\}\big|\to 0.\end{equation}
Moreover, this can in some cases be improved to convergence in Hausdorff distance, which can be interpreted as geometric uniform convergence. This type of convergence has been proved in \cite{ChaDuvPeyPoo17} for classical ROF denoising in the plane, and in \cite{IglMerSch18, IglMer20} for linear inverse problems, bounded domains, Banach space measurements and general dimensions. All of these results (with the exception of when explicit dual certificates for $u_{\alpha_n,0}$ are known \cite[Sec.~8]{ChaDuvPeyPoo17}) assume a source condition implying $\partial \TV(f) \neq \emptyset$ when $\TV$ is considered as a functional on $L^{d/(d-1)}(\R^d)$. On the one hand this source condition guarantees, in particular, that the level sets of the minimizers $u_{\alpha,w}$ satisfy uniform density estimates independent of $\alpha$ and $w$, as long as these are related through an adequate parameter choice. On the other, this subgradient condition excludes cases of interest where geometric convergence is still expected, like the case when $f$ is the indicatrix of a planar polygon \cite[Sec.~3.3]{ChaDuvPeyPoo17}.

Our main goal is to obtain this improved mode of convergence while assuming as little regularity of $\{ f > s \}$ as possible, and this is achieved in two different situations. The first is when $f$ is the characteristic function of a bounded finite perimeter set, and admitting noisy measurements with a natural parameter choice. The second concerns a generic class of $\BV$ functions in which ``flat regions are controlled'' and including piecewise constant functions, but with noiseless measurements. The techniques used have as a central point the variational mean curvatures for general finite perimeter sets introduced in \cite{BarGonTam87, Bar94} which, through comparison arguments, are used as a lower integrability replacement for the missing dual certificates for $f$.

\subsection{Main results}

\begin{theorem}\label{thm:localhausdorff}
Assume that $f=1_D$, the indicatrix of a bounded finite perimeter set $D \subset \R^d$, and that the sequences $\alpha_n \to 0$ and $w_n$ are such that there is some constant $C_{\psi, d}$ for which
\begin{equation}\label{eq:hardparamchoice}\frac{\|w_n\|_{L^{d/(d-1)}(\R^d)}}{\alpha_n} \ls C_{\psi, d} < \frac{m_{\psi^\ast}\left(\Theta_d\right)}{\Theta_d},\end{equation}
where $\Theta_d$ denotes the isoperimetric constant in $\R^d$ and $m_{\psi^\ast}$ is the largest modulus of uniform convexity (see Definition \ref{def:unifconv} below) for $\psi^\ast$. Then we have, up to a not relabelled subsequence, the convergence in Hausdorff distance
\begin{equation*}\dh(\partial \{u_{\alpha_n,w_n} > s\}, \partial D) \to 0 \text{ for a.e.~}s \in (0,1).\end{equation*}
\end{theorem}

\begin{theorem}\label{thm:generichausdorff}
Let $f \in \BV(\R^d)$ with bounded support. Denote $E_\alpha^s = \{u_{\alpha,0} > s\}$ and $G^s = \{f > s\}$ if $s>0$, and $E_\alpha^s = \{u_{\alpha,0} < s\}$ and $G^s = \{f < s\}$ for $s <0$. Assume that $s$ satisfies that $|E_\alpha^s \Delta G^s| \to 0$, and also
\begin{equation}\label{eq:assumhausdorff} \lim_{\nu \to 0} \dh\left(G^s, G^{s+\nu}\right) = 0, \text{ and }\lim_{\nu \to 0} \dh\left(\R^d \setminus G^s, \R^d \setminus G^{s+\nu}\right) = 0.\end{equation}
Then in the absence of noise ($w=0$) we have the convergence $\dh(\partial E_\alpha^s, \partial G^s) \to 0$.
\end{theorem}

\begin{remark} Any piecewise constant function satisfies trivially \eqref{eq:assumhausdorff} at all the values that it does not attain, and therefore the conclusions of Theorem \ref{thm:generichausdorff} are valid. It also holds for almost every $s$ under the source condition of \cite{ChaDuvPeyPoo17, IglMer20}, a fact we prove in Corollary \ref{cor:sourceimplies}. In contrast, Example \ref{ex:pingpongballs} provides a function $f$ for which the set of values where \eqref{eq:assumhausdorff} does not hold is of full measure.\end{remark}

\subsection{Structure of the paper}
We start with some preliminary results in Subsection \ref{subsec:preliminaries}. Section \ref{sec:hausdorff} is dedicated to auxiliary results about convergence in Hausdorff distance of bounded subsets of $\R^d$. In Section \ref{sec:curvatures} we are concerned with variational mean curvatures, with the main goal of pinning down the construction of such a curvature on the outside of any finite perimeter set, and also stating basic comparison results. Section \ref{sec:local} aims at the proof of Theorem \ref{thm:localhausdorff} using density estimates that degenerate as the set $D$ is approached and dual stability estimates with respect to the noise, which are themselves proved in Appendices \ref{sec:dual} and \ref{sec:densityest}. Then, Section \ref{sec:noiseless} is devoted to the proof of Theorem \ref{thm:generichausdorff} by approximation of finite perimeter sets with the level sets of their optimal variational mean curvatures. Finally, in Section \ref{sec:uniform} we explore whether it is possible to recover uniform density estimates without the source condition; it turns out that this is possible for indicatrices of some planar polygons.

\subsection{Preliminaries}\label{subsec:preliminaries}

The total variation $\TV$ appearing in \eqref{eq:ROFpsi} is the norm of the distributional derivative as a Radon measure, that is
\begin{equation*}\label{eq:defTV}\TV(u) := |Du|(\R^d) =  \sup \left\{\int_{\R^d} u\, \div z\, \dd x \, \middle\vert\, z \in \mathcal{C}^\infty_c(\R^d\, ;\,  \R^d), \|z\|_{L^\infty(\R^d)} \ls 1\right\}.\end{equation*}
Correspondingly we say that $u:\R^d \to \R$ is of bounded variation when it belongs to
\[\BV(\R^d):=\left\{ u \in L^1_\loc(\R^d) \,\middle\vert\, \TV(u) < +\infty \right\},\]
where we remark that we only require such functions to be locally summable. Likewise, the space $\BV_\loc(\R^d)$ consists of those functions $u \in L^1_\loc(\R^d)$ for which $|Du|(K)<+\infty$ for each compact set $K$. A set $E$ is called of finite perimeter whenever its indicatrix $1_E$ is of bounded variation, and the perimeter is defined as 
\[\per(E):=\TV(1_E).\]
Since this notion is invariant with respect to zero Lebesgue measure modification of $E$, we need a notion of boundary which satisfies this invariance as well. For this purpose, we can take a representative of $E$ for which the topological boundary equals the support of the derivative of $1_E$, which can be described \cite[Prop.~12.19]{Mag12} as
\begin{equation*}\label{eq:bdy}\partial E = \supp D1_E =\left\{ x \in \R^d \,\middle\vert\, 0< \frac{|E \cap B(x,r)|}{|B(x,r)|} < 1 \text{ for all } r>0\right\},\end{equation*}
and this choice will be assumed in all what follows. Notice that we might have $|\partial E| >0$ (see \cite[Example 12.25]{Mag12} for an example), so particular care is needed when combining topological and measure-theoretic arguments for this boundary. The topological interior of a set $E$ will be denoted by $\accentset{\circ}{E}$, and by $E^{(1)}$ its subset of points of full density in $E$. Moreover, convex hulls are denoted by $\conv E$ and complements by $E^c := \R^d \setminus E$.

\begin{definition}\label{def:subgrad}
Taking into account the Sobolev embedding $\BV(\R^d) \cap L^1(\R^d) \subset L^{d/(d-1)}(\R^d)$, we think of the total variation $\TV$ as defined on $L^{d/(d-1)}(\R^d)$ and with value $+\infty$ on $L^{d/(d-1)}(\R^d)\setminus \BV(\R^d)$. In this context, we speak of its subgradient at some function $u$ and denoted by $\partial \TV(u)$, to refer to the set of functions $v \in L^d(\R^d)$ such that
\begin{equation}\TV(\tilde u) - \TV(u) \gs \int_{\R^d} v (\tilde u - u) \dd x \text{ for all }\tilde u \in L^{\frac{d}{d-1}}(\R^d).\end{equation}
Let us note that by considering the functional in a larger space, we reduce the set of possible subgradients and avoid using the dual of the non-reflexive space $BV(\R^d)$.
\end{definition}

\begin{definition}\label{def:unifconv}
Let $\psi: \R \to \R$ be a convex function. We say that $\psi$ is uniformly convex \cite[Chapter 5.3]{BorVan10} with modulus of uniform convexity $h_\psi >0$ when for all $s,t \in \R$ and $0 \ls \mu \ls 1$ we have
\[\psi\left((1-\mu)s+\mu t)\right) \ls (1-\mu)\psi(s) + \mu\psi(t) - \mu(1-\mu)h_\psi\left(|s-t|\right).\]
Clearly there is a largest $h_\psi$ satisfying the last inequality, since it is stable under taking the maximum over two such functions.

Moreover, the function $\psi$ is said to be strictly convex when for all $s, t \in \R$ with $s \neq t$ and $0 < \mu < 1$ we have
\[\psi\left((1-\mu)s+\mu t)\right) < (1-\mu)\psi(s) + \mu\psi(t),\]
this property being weaker than uniform convexity.
\end{definition}

From strict convexity and general properties of the space $\BV(\R^d)$ one can deduce the following basic result on existence and convergence of minimizers for \eqref{eq:ROFpsi}:

\begin{proposition}\label{prop:existconv}
Assuming \eqref{eq:psiassum} and $\int \psi(w) < +\infty$, the minimization problem \eqref{eq:ROFpsi} admits a unique solution $u_{\alpha, w}$. Furthermore, if $\alpha_n \to 0$ and $w_n$ are such that \begin{equation}\label{eq:easyparamchoice}\frac{1}{\alpha_n} \int \psi(w_n) \ls C,\end{equation} then $u_{\alpha_n,w_n} \to f$ weakly in $L^{d/(d-1)}$ and strongly in $L^q_\loc$ for $1\ls q<d/(d-1)$.
\end{proposition}
\begin{proof}
Let $u_k$ be minimizing sequence for \eqref{eq:ROFpsi}. Discarding some elements of the sequence if necessary and using the symmetry of $\psi$ we have the estimate 
\begin{equation}\label{eq:fest}\frac{1}{\alpha}\int \psi(u_k-f-w) + \TV(u_k) \ls \frac{1}{\alpha}\int \psi(w) + \TV(f).\end{equation}
On the other hand we also have the Sobolev inequality \cite[Thm.~3.47]{AmbFusPal00}
\begin{equation}\label{eq:sobolevineq}\|u_k - c_k\|_{L^{d/(d-1)}} \ls C \TV(u_k)\end{equation}
for some constants $c_k \in \R$. Noticing that $\psi(t)>0$ for $t \neq 0$ and that $\int \psi(w)$ is finite, we must have $\left|\R^d \setminus \{|w|\gs \eps\}\right|< +\infty$ for all $\eps >0$. Using then \eqref{eq:fest} we have that $\int \psi(u_k-f-w)$ is also finite, so (because $f$ is compactly supported) the same is true for $\R^d \setminus \{|u_k|\gs \eps\}$, and since we are working with functions defined on all of $\R^d$ we conclude that $c_k=0$ for all $k$. Therefore, using weak-* compactness in $\BV$ \cite[Thm.~3.23]{AmbFusPal00} and weak compactness in $L^{d/(d-1)}$ we can extract a limit $u_{\alpha,w}$ in those topologies. Moreover, we have lower semicontinuity of $\TV$ with respect to $L^1_\loc$ convergence \cite[Rem.~3.5]{AmbFusPal00}, while positivity and convexity of $\psi$ implies that the first term of \eqref{eq:ROFpsi} is also lower semicontinuous with respect to weak $L^{d/(d-1)}$ convergence \cite[Thm.~3.20]{Dac08}, so $u_{\alpha, w}$ must be a minimizer of \eqref{eq:ROFpsi}, unique since $\psi$ is strictly convex.

In view of \eqref{eq:fest} and \eqref{eq:easyparamchoice}, one can apply the same compactness arguments to $u_{\alpha_n,w_n}$ to obtain a subsequence converging in weakly in $L^{d/(d-1)}$ and strongly in $L^q_\loc$. Moreover since $\psi$ is strictly convex, $\psi(0)=0$ and $\psi(t)>0$ if $t >0$ it must be increasing on $[0,+\infty)$, so \eqref{eq:easyparamchoice} implies that $w_n \to 0$ in measure, which in turn implies \cite[Thm.~2.30]{Fol99} also $w_n(x) \to 0$ for a.e.~$x$, up to possibly taking a further subsequence. Finally \eqref{eq:fest} also gives that $\int \psi(u_{\alpha_n, w_n}-f-w_n)\to 0$, so the limit must be $f$. Since for any subsequence we are able to find a further subsequence converging to the fixed limit $f$, the whole sequence $u_{\alpha_n,w_n}$ must converge to it.
\end{proof}

We recall that for any $E,F$ with finite perimeter we have \cite[Prop.~3.38(d)]{AmbFusPal00}
\begin{equation} \per(E \cap F) + \per(E \cup F) \ls \per(E) + \per(F),\label{eq:subaddper} \end{equation}
and the isoperimetric inequality \cite[Thm.~3.46]{AmbFusPal00}
\begin{equation}\per(F) \gs \Theta_d \min\left(|F|^{(d-1)/d}, |\R^d \setminus F|^{(d-1)/d}\right), \text{ with }\Theta_d = \frac{\per(B(0,1))}{|B(0,1)|^{(d-1)/d}}.\label{eq:isoper}\end{equation}

Below we study in detail problem \eqref{eq:ROFpsi} with $f=1_D$ and $w=0$, for which the minimizer $u_\alpha := u_{\alpha, 0}$ has level sets $E_\alpha^s := \{u_\alpha > s\}$ that minimize for a.e. $s \in (0,1)$
\begin{equation*}E \mapsto \per(E) + \frac{1}{\alpha} \int_E \psi'(s-f(x)) \dd x = \per(E) - \frac{\psi'(1-s)}{\alpha} |E \cap D| + \frac{\psi'(s)}{\alpha} |E \setminus D|,\end{equation*}
as can be seen from \eqref{eq:vawformula}, the coarea formula for $\BV$ functions \cite[Thm.~3.40]{AmbFusPal00} and the general layer cake formula \cite[Thm.~1.13]{LieLos01}. More generally we have:

\begin{proposition}\label{prop:slicedROFpsi}
 Let $u$ minimize \eqref{eq:ROFpsi}. Then for $s \in (0,+\infty)$ its upper level sets $E^s := \{u>s\}$ minimize, among sets of finite mass, the functional 
\begin{equation*}\label{eq:slicedprob}E \mapsto \per(E) + \frac{1}{\alpha}\int_E \psi'(s-f-w),\end{equation*}
and moreover we have
\begin{equation*}\label{eq:perimetereq}\per(E^s) = \frac{1}{\alpha} \int_{E^s} \psi'(f+w-s).\end{equation*}
For $s \in (-\infty,0)$ and the lower level sets $\{u<s\}$ analogous statements hold by changing the sign on the integral terms.
\end{proposition}
\begin{proof}
The proof of the first statement can be found in \cite[Prop.~2.3.14]{Jal12}. The second is proved in \cite[Prop.~3]{ChaDuvPeyPoo17}.
\end{proof}

\begin{remark}\label{rem:signs}Note that if $s<0$ it is necessary to work with the lower level sets $\{u<s\}$ so that $|\{u<s\}|<+\infty$, in which case the integral terms change sign. This will be useful to keep in mind in some results below.
\end{remark}

\section{Density estimates and Hausdorff convergence}\label{sec:hausdorff}
We begin with some auxiliary results on convergence in the Hausdorff distance, defined for $E,F \subset \R^d$ as
\begin{equation}\label{eq:hausdist}\dh(E,F):=\max\left\{\sup_{x \in E} \dist(x, F),\, \sup_{y \in F} \dist(y,E)\right\},\end{equation}
 and its relation with $L^1$ convergence when density estimates are available, which will be used in the proof of the main results.
\begin{definition}
Let $\{E_\gamma\}_\gamma$ be a family of finite perimeter sets of uniformly bounded measure, that is, there is $M>0$ such that $|E_\gamma| <M$ for all $\gamma$. If there are constants $r_0 >0$ and $C \in (0,1)$ such that for all $\gamma$ and all $x \in \partial E_\gamma$ we have for all $r < r_0$ that
\begin{equation}\label{eq:innerdensity}\frac{|E_\gamma \cap B(x,r)|}{|B(x,r)|} \gs C,\end{equation}
we say that this family satisfies \textit{uniform inner density estimates} with constant $C$ at scale $r_0$. Similarly, if instead we have for $r \ls r_0$
\begin{equation}\label{eq:outerdensity}\frac{|B(x,r) \setminus E_\gamma|}{|B(x,r)|} \gs C\end{equation}
we say that this family satisfies \textit{uniform outer density estimates}, again with constant $C$ at scale $r_0$. When speaking of \textit{uniform density estimates}, we understand that both estimates hold with the same constants.
\end{definition}

First, in \cite{IglMerSch18, IglMer20} the following result is claimed, although with some flaws in its presentation:
\begin{proposition}\label{prop:hausdorff_of_sets}Assume we have $\{E_n\}_n, E_0$ are subsets of $\R^d$ satisfying uniform inner density estimates with some scale $r_0$ and constant $C$, and such that $|E_n \Delta E_0| \to 0$. Then $\dh(E_n, E_0) \to 0$.
\end{proposition}
\begin{proof}
First, we notice that if we have the estimate 
\begin{equation}\label{eq:bdyinnerdens}|E_n \cap B(x,r)| \gs C |B(x,r)| \text{ for }x \in \partial E_n\text{ and }r \ls r_0,\end{equation}then we also have
\begin{equation}\label{eq:intinnerdens}|E_n \cap B(y,\tilde r)| \gs \frac{C}{2^d} |B(y,\tilde r)| \text{ for }y \in \overline{E_n}\text{, and }\tilde r \ls 2 r_0.\end{equation}
To see this, first set $r = \tilde r / 2$. Then, if $\dist(y, \partial E_n) \gs r$, the whole ball $B(y,r) \subset E_n$, so that $|E_n \cap B(y,\tilde r)| \gs |B(y,r)| = |B(y,\tilde r)|/2^d >C|B(y,\tilde r)|/2^d$ and \eqref{eq:intinnerdens} holds. If $0 \ls \dist(y, \partial E_{n}) < r$, then there is at least one boundary point $x_y \in \partial E_n$ for which $B(x_y, r) \subset B(y, \tilde r)$, and applying \eqref{eq:bdyinnerdens} to $x_y$ and $r$ we get \eqref{eq:intinnerdens}.

With these facts, let us assume that there is $\delta > 0$ such that $\dh(E_n, E_0) > \delta$ for infinitely many $n$, and derive a contradiction. Reducing $\delta$ if necessary, we can assume that $\delta \ls 2 r_0$. In view of the definition \eqref{eq:hausdist} we must then have a subsequence $n_k$ for which either $\sup_{x \in E_{n_k}} \dist(x, E_0) > \delta$ or $\sup_{x \in E_0} \dist(x, E_{n_k}) > \delta$. For the first case, we have a sequence of points $x_{n_k} \in E_{n_k}$ for which $\dist(x, E_0) > \delta$. Then \eqref{eq:intinnerdens} applied to $E_{n_k}$ and with $\tilde r = \delta$ gives
\[|E_{n_k} \Delta E_0| \gs |E_{n_k} \setminus E_0| \gs |E_{n_k} \cap B(x_{n_k},\delta)| \gs \frac{C}{2^d} \delta^d |B(0,1)|,\]
a contradiction with $|E_n \Delta E_0| \to 0$. For the second case, we obtain $x_{n_k} \in E_0$ for which $\dist(x_{n_k}, E_{n_k}) > \delta$. In this case, we use \eqref{eq:intinnerdens} for $E_0$ to end up as before with
\[|E_{n_k} \Delta E_0| \gs |E_0 \setminus E_{n_k}| \gs |E_0 \cap B(x_{n_k}, \delta)| \gs \frac{C}{2^d} \delta^d |B(0,1)|,\]
again a contradiction.
\end{proof}
In Proposition \ref{prop:hausdorff_of_sets} we only used the inner density estimates. However, for level sets of total variation minimizers and imaging applications one is mostly interested in convergence of their boundaries. The latter is not implied by the convergence of the sets themselves, even under other modes of convergence assumed, as demonstrated in the following example. We will see later that to obtain convergence of the boundaries, the outer density estimates also need to be used.

\begin{example}\label{ex:set_not_bdy}Consider the unit square $E_0:=(0,1)^2$ and a sequence obtained by removing from it thin triangles:
\[E_n:=(0,1)^2 \setminus \mathrm{Conv}\left(\left\{\left(\frac{1}{2}-\frac{1}{n+2}, 0\right),\ \left(\frac{1}{2}+\frac{1}{n+2}, 0\right),\ \left(\frac{1}{2}, \frac{1}{2}\right)\right\}\right),\]
which admits uniform inner density estimates, but with the outer densities not being uniform at $(1/2, 1/2)$. We have $|E_n \Delta E_0| \to 0$ and $D1_{E_n} \stackrel{\ast}{\rightharpoonup} D1_{E_0}$. To see the latter, just notice that $D1_{E_n} = \nu_{E_n} \mathcal{H}^1\mres\partial^\ast E_n$, and that each of the non-vanishing sides of the triangle converge to the same vertical segment, but with opposite orientations. Moreover, since $E_n \subset E_0$ also
\[\dh(E_n,E_0)=\sup_{x \in E_0}\dist(x,E_n) \ls \frac{1}{n+2} \to 0, \text{ but }\dh(\partial E_n, \partial E_0) = \frac{1}{2}.\]
\end{example}

\begin{remark}
In general, the Hausdorff distances $\dh(E, F)$ and $\dh(\partial E, \partial F)$ are not related. In \cite[Thm.~14]{Wil07} it is proven that these are equal for bounded closed convex sets, and in \cite[Examples 6 and 13]{Wil07} examples are given for pairs of planar sets where both possible strict inequalities hold.
\end{remark}

Under $L^1$ convergence, the Hausdorff convergence of boundaries is in fact stronger:
\begin{proposition}\label{prop:bdy_implies_set}Assume that $\{E_n\}_n, E_0$ are subsets of $\R^d$ such that we have the convergences 
\[|E_n \Delta E_0| \to 0\text{ and } \dh(\partial E_n, \partial E_0) \to 0.\]
Then also $\dh(E_n, E_0) \to 0$.
\end{proposition}
\begin{proof}
Assume that the hypotheses are satisfied but $\dh(E_n, E_0) \not\to 0$. Then there is $\delta >0$ with $\dh(E_{n_k}, E_0) > \delta$ for some subsequence $n_k$. Removing leading terms if needed, we can assume that 
\begin{equation}\label{eq:deltahalf}\dh(\partial E_{n_k}, \partial E_0) < \frac{\delta}{2}.\end{equation}
Now, we have
\[\dh(E_{n_k}, E_0) = \max\left(\sup_{x \in E_{n_k}} \dist(x, E_0), \sup_{x \in E_0} \dist(x, E_{n_k})\right) > \delta,\]
so at least one of the arguments in the supremum must be larger than $\delta$ for infinitely many $k$. Assume that it is the first one, and relabel the subsequence $n_k$ so that
\[\sup_{x \in E_{n_k}} \dist(x, E_0) > \delta,\]
implying that there is a sequence $x_{n_k} \in E_{n_k}$ for which $\dist(x_{n_k}, E_0) > \delta$. In consequence for all $y \in \overline{E_0}$, and in particular for all $y \in \partial E_0$, we have $|x_{n_k} - y| \gs \delta$. Therefore, for each $k$ we must have $\dist(x_{n_k}, \partial E_{n_k}) \gs \delta/2$, since otherwise \eqref{eq:deltahalf} and the triangle inequality would lead to a contradiction. Because of the last inequality, the ball $B(x_{n_k},\delta/3)$ cannot intersect $\partial E_{n_k}$ and in consequence also not $\R^d \setminus E_{n_k}$, which forces
\[B\left(x_{n_k}, \frac{\delta}{3}\right) \subset E_{n_k} \text{ and }d\left(x_{n_k}, E_0\right) > \delta,\text{ so }B\left(x_{n_k}, \frac{\delta}{3}\right) \subset E_{n_k}\setminus E_0,\]
a contradiction with $|E_{n} \Delta E_0| \to 0$. The other case is dealt with similarly.
\end{proof}

\begin{proposition}\label{prop:hausineq}Let $E,F \subset \R^d$. Then
\begin{equation}\label{eq:hausineq}\dh(\partial E, \partial F) \ls \max \left(\dh(E, F), \dh(E^c, F^c)\right).\end{equation}
\end{proposition}
\begin{proof}
We use the characterization (often used as definition of $\dh$, see for example  \cite[Sec.~4.C]{RocWet98})
\begin{equation}\label{eq:hausdorffdilation}\dh(A, B)=\inf \left\{ r\gs0 \,\middle\vert\, A\subset U_rB \text{ and }B\subset U_rA \right\},\end{equation}
for the dilations $U_rA= \{\dist(\cdot, A)\ls r\}$.
Now, if the inequality to be proved failed, denoting 
\[r=\max \left(\dh(E, F), \dh(E^c, F^c)\right)\] 
we would have that either $U_r\partial F \setminus \partial E \neq \emptyset$ or $U_r\partial E \setminus \partial F \neq \emptyset$. Without loss of generality assume that the first case holds, so that there is $x \in \partial E$ for which $\dist(x, \partial F)>r$. If $x \in F^c$, by the properties of the boundary we can produce $\hat x \in E \setminus F$ with $\dist(\hat x, \partial F)>r$, which since $\hat x \in F^c$ also implies 
\[\dist(\hat x, F)=\dist(\hat x, \partial F)>r,\] 
contradicting $r \gs \dh(E, F)$. Similarly, if $x \in \{\dist(\cdot, \partial F)>r\}\cap F$, we can find $\check x \in E^c$ with $\dist(\check x, \partial F)>r$ as well, and as before since $\check x\in F \setminus E$ we have 
\[\dist(\check x, F^c)=\dist(\check x, \partial F)>r,\]
a contradiction with $r \gs \dh(E^c, F^c)$.
\end{proof}

Combining Propositions \ref{prop:bdy_implies_set} and \ref{prop:hausineq} we obtain

\begin{theorem}\label{thm:hausdorffequiv}
Assume that $\{E_n\}_n, E_0$ are subsets of $\R^d$ such that $|E_n \Delta E_0| \to 0$. Then $\dh(\partial E_n, \partial E_0) \to 0$ if and only if $\dh(E_n, E_0) \to 0$ and $\dh(E_n^c, E_0^c) \to 0$ simultaneously.
\end{theorem}

We can conclude Hausdorff convergence of the boundaries without the need of derivatives, by using both density estimates:

\begin{theorem}\label{thm:bdyhausdorff}Assume $\{E_n\}_n, E_0$ are subsets of $\R^d$ satisfying uniform density estimates with some scale $r_0$ and constant $C$, and such that $|E_n \Delta E_0| \to 0$. Then
\begin{equation*}\dh(\partial E_n, \partial E_0) \to 0.\end{equation*}
\end{theorem}
\begin{proof}
We notice that
\begin{equation*}\label{eq:symdiffcomp}\begin{aligned}E_n \Delta E_0 &= \left(E_n \setminus E_0\right) \cup \left(E_0 \setminus E_n\right) = \left(E_n \cap E_0^c\right) \cup \left(E_0 \cap E_n^c\right)\\
&= \left(E_0^c \setminus E_n^c\right) \cup \left(E_n^c \setminus E_0^c\right) = E_o^c \Delta E_n^c,
\end{aligned}\end{equation*}
and that by taking complements the roles of \eqref{eq:innerdensity} and \eqref{eq:outerdensity} are reversed. Therefore, using both we can apply Proposition \ref{prop:hausdorff_of_sets} for $E_n$ and for $E_n^c$ so that $\dh(E_n, E_0) \to 0$ and $\dh(E_n^c, E_0^c)\to 0$. Proposition \ref{prop:hausineq} gives then the conclusion.
\end{proof}

As a direct consequence we get the following result, proved but not explicitly stated in \cite{ChaDuvPeyPoo17}, which also applies to the cases treated in \cite{IglMerSch18, IglMer20}:
\begin{proposition}\label{prop:hausdorff_of_bdy}Assume we have $\{E_n\}_n, E_0$ finite perimeter sets satisfying uniform density estimates with some scale $r_0$ and constant $C$, and such that the characteristic functions $1_{E_n} \stackrel{\ast}{\rightharpoonup} 1_E$ in $\BV$. Then $\dh(\partial E_n, \partial E_0) \to 0$.
\end{proposition}

\section{A few results on variational mean curvatures}\label{sec:curvatures}

We now turn our attention to the weak notion of mean curvature for boundaries which will be our main tool to describe the behaviour of level sets of minimizers of \eqref{eq:ROFpsi}:

\begin{definition}
We say that a set $A \subset \R^d$ has a variational mean curvature $\kappa: \R^d \to \R$ if it minimizes, among $E \subset \R^d$, the functional
\begin{equation}E \mapsto \per(E) - \int_E \kappa.\label{eq:varCurvE}\end{equation}
\end{definition}

If the set $A$ has a smooth boundary and $\kappa$ is continuous, this minimization property implies that the restriction of $\kappa$ to the boundary $\partial A$ is, up to a multiplicative factor, the usual mean curvature of $\partial A$. To see this, just notice \cite[Rem.~17.6]{Mag12} that if $\partial A$ is $\mathcal{C}^2$, the first variation of the perimeter along the flow generated by a vector field $V \in \mathcal{C}_c^\infty(\R^d ;\R^d)$ is
\begin{equation}\label{eq:intbypartsH}\int_{\partial A} \div_{\partial A} V \dd\mathcal{H}^{d-1} = \int_{\partial A} (d-1)H_A \, V \!\cdot \nu_A \dd\mathcal{H}^{d-1},\end{equation}
for $\div_{\partial A}$ the surface divergence, $\nu_A$ the outward normal vector and $H_A$ the usual mean curvature of $\partial A$, while that of the integral term in \eqref{eq:varCurvE} for continuous $\kappa$ amounts to
\[-\int_{\partial A} \kappa \, V \!\cdot \nu_A \dd\mathcal{H}^{d-1},\]
from which we conclude by noticing that $A$ is a minimizer of \eqref{eq:varCurvE} and $T$ is arbitrary. Analogously, if we had $u \in \mathcal{C}^2$ a minimizer of \eqref{eq:ROFpsi} with $w=0$ and $f$ continuous, using the implicit function theorem and Proposition \ref{prop:slicedROFpsi}, we would find for $E^s := \{u>s\}$ that $H_{E^s}=-\psi'(s-f)/\alpha$. If additionally $\nabla u(x) \neq 0$ for all, $x$ taking the first variation of $\TV(u)$, which under this assumption is differentiable and equals $\int |\nabla u|$, leads to
\[-\frac{1}{\alpha}\psi'\big(u(x)-f(x)\big)=(d-1)H_{E^s}(x)=\div\left(\frac{\nabla u(x)}{|\nabla u(x)|}\right) \text{ for }x \in \partial E^s, \text{ so }u(x)=s.\]
We recall that there is a natural weak notion of mean curvature based on \eqref{eq:intbypartsH}, the \emph{distributional} mean curvature, which can be defined not just for boundaries of finite perimeter sets but also for most notions of non-regular surfaces (e.g. varifolds). The distributional and variational mean curvatures coincide in the very regular case just described, but it is not quite clear whether they do on less regular cases where both are available; some positive results are given in \cite{BarGonMas03}.

From the definition one sees that variational mean curvatures for a given set, as functions defined in $\R^d$, contain ``too much information" and one can not expect them to be unique. In fact, if $\kappa$ is a variational curvature for $A$, any other function $\kappa'$ with $\kappa' \gs \kappa$ on $A$ and $\kappa' \ls \kappa$ on $\R^d \setminus A$ is another variational mean curvature for $A$ as well.

\begin{remark}Using the coarea and layer-cake formulas as for Proposition \ref{prop:slicedROFpsi}, it is straightforward to check that if we have $v \in \partial\TV(f)$ for some $f \in L^{d/(d-1)}$ and $v \in L^d$, almost all of the upper level sets of $f$ at positive values are minimizers of \eqref{eq:varCurvE} with $\kappa = v$, making $v$ a variational curvature for all of them. For negative values, one switches to lower level sets and the curvature sign to $-v$, cf. Remark \ref{rem:signs}.\end{remark}

Being a minimizer of \eqref{eq:varCurvE} is stable by intersection and union, which in particular enables speaking about maximal and minimal minimizers:

\begin{proposition}\label{prop:cupcapofmin}
Let $E_1$ and $E_2$ be two minimizers of \eqref{eq:varCurvE}. Then $E_1 \cap E_2$ and $E_1 \cup E_2$ are also minimizers of \eqref{eq:varCurvE}. In particular if one has at least one minimizer (possibly with inclusion constraints) of \eqref{eq:varCurvE}, then there are also a maximal minimizer and a minimal one with respect to inclusion.
\end{proposition}
\begin{proof}
 One can write, using the minimality of $E_1$ and $E_2$ and noting that $E_1 \cap E_2$ as well as $E_1 \cup E_2$ are admissible for that problem,
 \begin{equation}\label{eq:capopt} \per(E_1 \cap E_2) - \int_{E_1 \cap E_2} \kappa \gs \per(E_1) - \int_{E_1} \kappa,\end{equation}
 and
  \begin{equation}\label{eq:cupopt} \per(E_1 \cup E_2) - \int_{E_1 \cup E_2} \kappa \gs \per(E_2) - \int_{E_2} \kappa. \end{equation}
Summing these inequalities and noticing that the volume terms exactly compensate, we obtain
\begin{equation}\label{eq:reversefundamental}\per(E_1 \cap E_2) + \per(E_1 \cup E_2) \gs \per(E_1) + \per(E_2).\end{equation}
Now, if either of the inequalities \eqref{eq:capopt} or \eqref{eq:cupopt} were strict, we would also have a strict inequality in \eqref{eq:reversefundamental}. But this would be a contradiction with \eqref{eq:subaddper}, so all of these inequalities must be equalities.
\end{proof}

We make extensive use of the following basic but fundamental comparison lemma for variational mean curvatures:

\begin{lemma}\label{lem:basiccomp}Assume that the finite perimeter sets $E_1$ and $E_2$ admit variational mean curvatures $\kappa_1$ and $\kappa_2$ respectively, and such that $\kappa_1 < \kappa_2$ in $E_1 \setminus E_2$. Then $|E_1\setminus E_2|=0$, that is $E_1 \subseteq E_2$ up to Lebesgue measure zero.
\end{lemma}
\begin{proof}
We can write
\[\per(E_1) - \int_{E_1}\kappa_1 \ls \per(E_1 \cap E_2) - \int_{E_1 \cap E_2} \kappa_1,\]
\[\per(E_2) - \int_{E_2}\kappa_2 \ls \per(E_1 \cup E_2) - \int_{E_1 \cup E_2} \kappa_2.\]
Summing and using \eqref{eq:subaddper}, we arrive at 
\[\int_{E_1 \setminus E_2} \kappa_2 \ls \int_{E_1 \setminus E_2} \kappa_1,\]
which implies the result.
\end{proof}

We will repeatedly use the previous lemma to compare with balls:

\begin{example}\label{ex:curvball}
For $x_0 \in \R^d$ and $r>0$, any function $v_{B(x_0, r)} \in L^1(\R^d)$ with
\begin{equation}\begin{gathered}v_{B(x_0, r)}=\frac{d}{r} \text{ in }B(x_0, r),\ \, v_{B(x_0, r)} < 0 \text{ in }\R^d\setminus B(x_0, r),\text{ and }\\ \int_{\R^d\setminus B(x_0, r)} v_{B(x_0, r)} = -\per(B(x_0, r))\label{eq:intoutsideper}\end{gathered}\end{equation}
is a variational mean curvature for $B(x_0,r)$. To check this, first we notice that
\[\per(B(x_0, r)) - \int_{B(x_0, r)} v_{B(x_0, r)}(x) \dd x = 0.\]
Moreover, for any other finite perimeter set with $|E| < \infty$, we have by the isoperimetric inequality \eqref{eq:isoper} that for arbitrary $y \in \R^d$
\[\per\left(B\left(y, r_E\right)\right) \ls \per(E), \text{ with }r_E:=\left(\frac{|E|}{|B(0,1)|}\right)^{1/d},\]
and clearly $|B(y, r_E) \cap B(x_0, r)|$ is maximized by picking $y = x_0$. If $r_E > r$ then $\per(B(x_0, r_E)) > \per(B(x_0, r_0))$ but
\[\int_{B(x_0, r_E)} v_{B(x_0, r)}(x) \dd x < \int_{B(x_0, r)} v_{B(x_0, r)}(x) \dd x,\]
so $E$ could not be a minimizer. If $r_E \ls r$, then
\begin{align*}\per(B(x_0, r_E)) &= \left(\frac{r_E}{r}\right)^{d-1} \per(B(x_0, r)) = \left(\frac{r_E}{r}\right)^{d-1} \int_{B(x_0, r)} v_{B(x_0, r)}(x) \dd x \\
&= \frac{r}{r_E} \int_{B(x_0, r_E)} v_{B(x_0, r)}(x) \dd x \gs \int_{B(x_0, r_E)} v_{B(x_0, r)}(x) \dd x,\end{align*}
with equality if and only if $r_E = r$. The case in which $|\R^d \setminus E| < +\infty$, implying that $E$ must be of the form $\R^d \setminus B(y,\tilde{r}_E)$ for some $\tilde{r}_E >0$, is handled with similar computations once we notice that condition \eqref{eq:intoutsideper} prevents the full space $\R^d$ from having negative energy.
\end{example}

Furthermore, Lemma \ref{lem:basiccomp} combines with the strict convexity of $\psi$ to give a comparison principle for denoised solutions:

\begin{proposition}\label{prop:ROFpsicomparison}Let ${g} \ls {f}$ and $u^f_{\alpha,0}$, $u^g_{\alpha,0}$ be the corresponding minimizers of \eqref{eq:ROFpsi} with $w=0$. Then one has $u^{{g}}_{\alpha,0} \ls u^{{f}}_{\alpha,0}$.
\end{proposition}
\begin{proof}
To simplify the notation we drop the subindices that remain arbitrary, but fixed, in what follows. By Proposition \ref{prop:slicedROFpsi}, one can see that the level sets $\{u^f \gs s\}$ and $\{u^g \gs s\}$ are the maximal minimizers (with respect to inclusion) among $E$ of respectively
\[  \per(E) + \frac{1}{\alpha}\int_E \psi'(s-f), \enskip \text{and} \enskip \per(E) + \frac{1}{\alpha}\int_E \psi'(s-g).\]
Since $\psi$ is strictly convex, we then have, for $s' < s$,
\[\psi'(s-{g}) > \psi'(s'-{f}),\]
which implies by Lemma \ref{lem:basiccomp} that $|\{ u^g \gs {s}\} \setminus \{ u^f \gs {s'}\}|=0$. Since $s'<s$ was arbitrary and these sets are nested with respect to $s'$, we infer
\[|\{ u^g \gs s\} \setminus \{ u^f \gs s\}|=\left| \,{\bigcup_{n} \left(\{ u^g \gs s\} \setminus \Big\{ u^f \gs s-\frac{1}{n}\Big\}\right)}\right| = 0.\]
Denoting the set
\[A:=\bigcup_{s \in \R}\{ u^g \gs s\} \setminus \{ u^f \gs s\} = \bigcup_{s \in \R}\{ u^g \gs s\} \cap \{ u^f < s\},\]
we would like to see that $|A|=0$ so that $u^{{f}} \gs u^{{g}}$ almost everywhere. We cannot immediately conclude since the union is over an uncountable index set. To proceed, define
\[A_{\mathbb{Q}}:= \bigcup_{r \in \mathbb{Q}}\{u^g \gs r\} \cap \{u^f < r\}\]
with $|A_{\mathbb{Q}}|=0$, and let $x \in A \setminus A_{\mathbb{Q}}$. Then there is some $s_0 \in \R$ for which both $u^g(x) \gs s_0$ and $u^f(x) < s_0$ hold. However, for all $r \in \mathbb{Q}$ we have either $u^g(x) < r$ or $u^f(x) \gs r$. Let $\{r_n\}_n \subset \mathbb{Q}$ with $r_n < s_0$ and $r_n \to s_0$. If we had that $u^g(x)<r_n$ for some $n$, then $u^g(x)<r_n<s_0$, a contradiction. So we must have $u^f(x) \gs r_n$ for all $n$, implying that $u^f(x) \gs s_0$, which is again a contradiction. Therefore $A=A_{\mathbb{Q}}$.
\end{proof}

In consequence, we also have
\begin{corollary}\label{cor:nonewlevels}
If $g$ has values in $[a,b]$, then for every $\alpha$ also $u^g_{\alpha, 0}$ has values in $[a,b]$.
\end{corollary}
\begin{proof}
Notice that to consider the minimization of \eqref{eq:ROFpsi} we do not require that the data $f$ (with $w=0$) is in $L^{d/(d-1)}(\R^d)$. It is enough that there is a constant $c_f$ such that $f-c_f \in L^{d/(d-1)}(\R^d)$ and the corresponding solution will also have this property, see the usage of the Sobolev inequality \eqref{eq:sobolevineq} in the proof of Proposition \ref{prop:existconv}. This allows to compare with constant functions, which are invariant by minimizing \eqref{eq:ROFpsi}.
\end{proof}

\subsection{Construction of variational mean curvatures for bounded sets} 

A natural question is whether a variational mean curvature can be found for a given set. The following crucial result proven in \cite{BarGonTam87, Bar94} provides a positive answer:

\begin{theorem}\label{thm:optimalcurvature}
Let $D$ be a bounded set with finite perimeter. Then, $D$ has at least one variational mean curvature in $L^1(\R^d)$. In addition, there exists a variational mean curvature $\kappa_D$ for $D$ which minimizes the $L^p(D)$ norm for all $p>1$ among such curvatures. There might be $p>1$ for which this minimal norm is not finite.
\end{theorem}

The construction of $\kappa_D$ in \cite{BarGonTam87, Bar94} involves choosing a positive function $g \in L^1(\R^d)$ and minimizers of the problems
\begin{equation}\label{eq:insideprob}\min_{E \subset D} \,\per(E) - \lambda \int_E g, \text{ and }\end{equation}
\begin{equation}\label{eq:outsideprob}\min_{F \subset \R^d \setminus D} \,\per(F) - \lambda \int_F g.\end{equation}
Namely, for $x \in D$ one defines
\begin{equation}\label{eq:kappaplus}\kappa_D(x):=\inf\left\{\lambda g(x) \ \middle\vert \ \lambda >0 \text{ and } x \in E^\lambda, \text{ for }E^\lambda \text{ any minimizer of }\eqref{eq:insideprob}\right\},\end{equation}
and for $x \in \R^d \setminus D$
\begin{equation}\label{eq:kappaminus}\kappa_D(x):=-\inf\left\{ \lambda g(x) \ \middle\vert \ \lambda>0 \text{ and } x \in F^\lambda, \text{ for }F^\lambda \text{ any minimizer of }\eqref{eq:outsideprob}\right\}.\end{equation}
By definition $\kappa_D >0$ in $D$ and $\kappa_D < 0$ in $\R^d \setminus D$, consistent with the lack of uniqueness for variational mean curvatures described above. In later sections of this article we require a specific choice of $g$ made precise in Definition \ref{def:gR}. Moreover, the proof of Proposition \ref{prop:cupcapofmin} is also valid with inclusion constraints, so one can speak of maximal and minimal $E^\lambda$ and $F^\lambda$. For completeness, we check that $\kappa_D$ is well defined:

\begin{proposition}\label{prop:kappawelldef}
The problems \eqref{eq:insideprob} and \eqref{eq:outsideprob} admit at least one minimizer. Moreover if for every compact set $K \subset \R^d$ one can find $c_K$ such that 
\begin{equation}\label{eq:locallowerbound}g(x) \gs c_K >0 \text{ for a.e. }x \in K,\end{equation}
then for almost every $x \in D$, we have that $x \in E^{\lambda_x}$ for some $\lambda_x>0$ and $E^{\lambda_x}$ a minimizer of \eqref{eq:insideprob}, and similarly for a.e.~$x \in \R^d \setminus \overline D$ and a corresponding minimizer of \eqref{eq:outsideprob}.
\end{proposition}
\begin{proof}
Let us focus first on minimizers of \eqref{eq:outsideprob}, for which we consider the equivalent problem for the complement
\[\min_{E \supset D} \per(E) +\lambda\int_E g.\]
Let $\{E_n\}_n$ be a minimizing sequence for this problem. The objective is nonnegative, so comparing with any fixed nonempty set we have an upper bound for $\per(E_n)$. Since $1_{E_n}(x) \in \{0,1\}$ and hence bounded in $L^1_\loc$, we can then apply compactness in $\BV_\loc$ \cite[Thm.~3.23]{AmbFusPal00} to obtain $v \in \BV_\loc(\R^d)$ such that $1_{E_n}\to v$ in $L^1_\loc$ and in consequence also almost everywhere. Since $\per(E_n)=|D1_{E_n}|(\R^d) \ls C$ we have in fact that $v \in \BV(\R^d)$, and since the convergence is in $L^1_\loc$ strong, there must be a finite perimeter set $E_0$ for which $v = 1_{E_0}$. Furthermore, by the lower semicontinuity of the total variation \cite[Rem.~3.5]{AmbFusPal00} with respect to $L^1_\loc$ convergence and since $g \in L^1(\R^d)$, using the dominated convergence theorem we have
\[\per(E_0) \ls \liminf_n \per(E_n), \text{ and } \int_{E_0} g = \liminf_n \int_{E_n} g,\]
so $\R^d \setminus E_0$ is a minimizer of \eqref{eq:outsideprob}. For \eqref{eq:insideprob} one proceeds similarly, with the difference that under the constraint $E \subset D$ the fact that $D$ is bounded allows to obtain full $L^1$ convergence of a minimizing sequence $\{E_n\}_n$.

To see the second part, we treat the inside and outside problems separately. First, notice that $D$ is admissible in \eqref{eq:insideprob}, so we have that
\[\per(E^\lambda)-\lambda \int_{E^\lambda} g \ls \per(D)-\lambda \int_D g,\]
or equivalently
\[\lambda \left( \int_D g - \int_{E^\lambda} g \right) \ls -\per({E}^\lambda)+\per(D)\ls \per(D),\]
where since $g >0$ and $E^\lambda \subset D$ the left hand side is positive, and using \eqref{eq:locallowerbound} for $\overline{D}$ we get
\[ |D \setminus E^\lambda| \ls \frac{1}{c_{\overline{D}}}\left( \int_D g - \int_{E^\lambda} g \right) \xrightarrow[\lambda \to \infty]{} 0,\]
so for a.e.~$x \in D$ we must have $x \in E^{\lambda_x}$ for some $\lambda_x$. Similarly $\R^d \setminus D$ is admissible in \eqref{eq:outsideprob}, so using $\per(\R^d \setminus D) = \per(D)$ we have for $F^\lambda$ any minimizer of \eqref{eq:outsideprob} the bound
\[\lambda \left( \int_{\R^d \setminus D} g - \int_{F^\lambda} g \right) \ls \per(D).\]
This time, to be able to use \eqref{eq:locallowerbound} we would need to see that $(\R^d \setminus D)\setminus F^\lambda = \R^d \setminus F^\lambda$ is bounded, which is not a priori obvious. For large enough $\lambda$ we prove in Lemma \ref{lem:boundedcomplement} below that $\R^d \setminus F^\lambda$ is indeed bounded, allowing us to conclude.
\end{proof}

\begin{lemma}\label{lem:boundedcomplement}
Assume that for every compact set $K \subset \R^d$ one can find $c_K$ such that \eqref{eq:locallowerbound} holds, and that $D \subset B(0,1)$. Then there is some $\lambda_1$ such that if $\lambda > \lambda_1$ all minimizers $F^\lambda$ of \eqref{eq:outsideprob} satisfy $\R^d \setminus F^\lambda \subset \overline{B(0,1)}$, and in particular $|\R^d \setminus F^\lambda| < +\infty$. Moreover in that case $F^\lambda \cap B(0,2)$ is also a minimizer of
\begin{equation}\label{eq:neumannextprob}\min_{F \subset B(0,2) \setminus D} \,\per\big(F;\,B(0,2)\big) - \lambda \int_F g.\end{equation}
\end{lemma}
\begin{proof}
Let us define the compact set
\[K_1:=\overline{B(0,2)} \setminus B(0,1) = \bigcup_{x \in \partial B(0,3/2)} \overline{B\left(x,1/2\right)}.\]
Using Lemma \ref{lem:basiccomp}, Example \ref{ex:curvball}, this expression and the condition on $g$ we see that if 
\[\lambda > \frac{2d}{c_{K_1}} =: \lambda_1\text{, then }\accentset{\circ}{K}_1 = B(0,2) \setminus \overline{B(0,1)} \subset F^\lambda.\] 
But since $g>0$ and $\partial B(0, 3/2) \subset \accentset{\circ}{K}_1$ this means that
\[\per\left(F^\lambda \cup \big(\R^d \setminus B(0,3/2)\big)\right) \ls \per(F^\lambda) \text{ and }\int_{F^\lambda \cup \,(\R^d \setminus B(0,3/2))} g \gs \int_{F^\lambda} g,\]
so necessarily $\R^d \setminus B(0,3/2) \subset F^\lambda$ for all $\lambda > \lambda_1$ as well, hence $\R^d \setminus \overline{B(0,1)} \subset F^\lambda$. These considerations also directly prove that $F^\lambda \cap B(0,2)$ minimizes \eqref{eq:neumannextprob}.
\end{proof}

The curvatures arising from this construction are in fact not independent of the choice of the density $g$, as is shown in Proposition \ref{prop:emptyFlambda} below. As has been noted in previous works \cite{Bar94, IglMer20}, since we work with bounded $D$, this ambiguity can be mitigated by choosing $g(x) = 1$ for all $x \in D$. However $g \in L^1(\R^d \setminus D)$ is required to make sense of the unbounded problem \eqref{eq:outsideprob}, and there is no canonical choice for it outside of $D$. Moreover, as opposed to most other works using this variational mean curvature, we plan to make explicit use of $\kappa_D$ on $\R^d \setminus D$ and the corresponding minimizers of \eqref{eq:outsideprob}.

\begin{proposition}\label{prop:emptyFlambda}For any bounded $D$ and any positive $g \in L^1(\R^d)$, there exists some $\lambda_g > 0$ such that if $\lambda \ls \lambda_g$, the only minimizer of \eqref{eq:outsideprob} is the empty set, and in consequence $\kappa_D(x) \gs -\lambda g(x)$ for a.e.~$x$.  Moreover, if additionally $|\R^d \setminus F^\lambda| < +\infty$ for all $\lambda$ then
\begin{equation*}\label{eq:kappafar}\kappa_D(x) = - \lambda_g \, g(x) \text{ for a.e. }x \in \R^d \setminus \conv D.\end{equation*}
\end{proposition}
\begin{proof}
Let us define
\[G_D := \argmin_{E \supset D}\per(E),\]
which exists by the same compactness arguments as in Proposition \ref{prop:kappawelldef} (if $d=2$ then in fact $G_D=\conv D$ \cite{FerFus09}). Then any minimizer $F \neq \emptyset$ of \eqref{eq:outsideprob} must have $\per(F) \gs \per(G_D)$, so that
\begin{equation}\label{eq:convest}\per(F)-\lambda\int_F g \gs \per(G_D) -\lambda \int_{\R^d \setminus D} g.\end{equation}
But whenever 
\[\lambda < \lambda_c := \frac{\per(G_D)}{\int_{\R^d \setminus D} g}\]
we have that the right hand side of \eqref{eq:convest} is positive, making $F$ a worse competitor than the empty set. We can then define
\[\lambda_g := \sup\left\{ \lambda>0 \,\, \middle\vert\,\, \inf_{F \subset \R^d \setminus D} \,\per(F) - \lambda \int_F g = 0 \right\} \gs \lambda_c >0.\]
Note that $\inf_{F \subset \R^d \setminus D} \,\per(F) - \lambda \int_F g \ls \per(D) - \lambda \int_{\R^d \setminus D}g ,$
the latter being negative as soon as $\lambda \int_{\R^d \setminus D}g < \per(D)$. This implies that \[\lambda_g \ls \frac{\per(D)}{\int_{\R^d \setminus D}g }.\]

To prove the second part, notice that having $\kappa_D(x) < - \lambda_g \, g(x)$ means that $x \notin F^{\lambda_g + \eps}$ for some $\eps >0$, or equivalently, that $x$ belongs to the minimal (in the sense of Proposition \ref{prop:cupcapofmin}) minimizer $E^{\lambda_g + \eps}$ of
\[\min_{E \supset D} \per(E) +(\lambda_g + \eps) \int_E g.\]
However since by assumption we have $|E^{\lambda_g + \eps}| = |\R^d \setminus F^{\lambda_g + \eps}| < \infty$, taking its intersection with a convex set cannot increase the perimeter (see \cite[Lem.~3.5]{BriTor15} for a proof in the general setting) and we get
\[\per(E^{\lambda_g + \eps})+(\lambda_g+\eps)\int_{E^{\lambda_g + \eps}}g \gs \per(E^{\lambda_g + \eps}\cap \conv D)+(\lambda_g+\eps)\int_{E^{\lambda_g + \eps}\cap \conv D}g,\]
and the inequality would be strict if $|E^{\lambda_g + \eps} \setminus \conv D|>0$, so necessarily
\[\left| \left\{x \in \R^d \setminus \conv D \,\middle|\, \kappa_D(x) < - \lambda_g \, g(x) \right\}\right| = \left| \bigcap_{\eps >0} E^{\lambda_g + \eps} \setminus \conv D \right| = 0.\]
\end{proof}

We see that the concrete choice of density $g$ affects the values of $\kappa_D$. We now introduce one such choice which at least allows for a purely geometric description of minimizers for $\lambda$ large enough.

\begin{definition}\label{def:gR}
Assume $D \subseteq B(0,1)$. For any $R > 1$ we define $g_R$ by 
\begin{equation}\label{eq:defgR}g_R(x):=\begin{cases}1 &\text{ if }0\ls |x| \ls R,\\ g_f &\text{ if }|x|>R,\end{cases}\end{equation}
for some $g_f \in L^1(\R^d \setminus B(0,R))$ with $0 < g_f \ls 1$ and satisfying \eqref{eq:locallowerbound}.
\end{definition}

\begin{figure}
\centering
 \includegraphics[width = 0.7\textwidth]{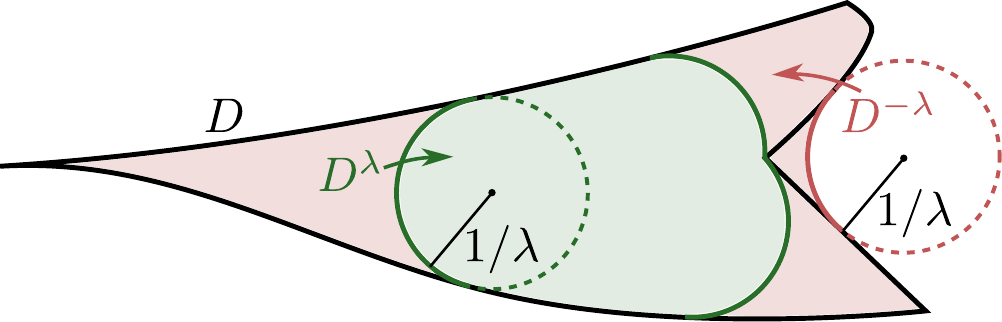}
 \caption{Approximation of the set $D$ with the level sets of $\kappa_D$. For fixed $\lambda >0$, we always have $D^\lambda \subset E \subset D^{-\lambda}$, and $\dh(\partial D^\lambda, \partial D^{-\lambda}) \to 0$ as $\lambda \to +\infty$. With respect to their outer normals, the free boundaries of $D^\lambda$ and $D^{-\lambda}$ have curvature $\lambda$ and $-\lambda$ respectively.} 
 \label{fig:approximation}
\end{figure}

Since we will make extensive use of minimizers of \eqref{eq:insideprob} and \eqref{eq:outsideprob} with this particular choice of density, we introduce some notation for them.
\begin{definition}
\label{def:dlambda}
Let $\lambda >0$ and $D \subset B(0,1)$ be of finite perimeter. We denote by $D^\lambda$ the maximal (in the sense of inclusion) minimizer of \eqref{eq:insideprob} with density $g_2$, that is, of
\begin{equation}\label{eq:Dlambdaplusprob1}\min_{E \subset D} \per(E) - \lambda |E|.\end{equation} 
We also define $D^{-\lambda}$ as
\[D^{-\lambda} := \R^d \setminus F^\lambda,\]
where $F^\lambda$ is the maximal minimizer of \eqref{eq:outsideprob} with density $g_2$. In view of the proof of Lemma \ref{lem:boundedcomplement}, whenever $\lambda > 2d$ we have that $F^\lambda$ can be determined from \eqref{eq:neumannextprob}, which since $g_2 \equiv 1$ on $B(0,2)$ turns into
 \begin{equation*}
  \label{eq:outsideprob1}
  \min_{F \subset B(0,2) \setminus D} \,\per\big(F;\,B(0,2)\big) - \lambda |F|.
 \end{equation*} 
Moreover, notice that $D^{-\lambda}$ can also be found directly as the minimal minimizer of
 \begin{equation*}
  \label{eq:Dlambdaminusprob}
  \min_{E \supset D} \,\per(E) + \lambda \int_E g_2.
 \end{equation*}
\end{definition}

\begin{remark}\label{rem:rescaling}
We have chosen $D \subset B(0,1)$ but other bounded sets can be treated by rescaling. If for any set $E$ and $q>0$ we consider the rescaled set $qE$ we have
\[\per(qE)-\lambda \int_{qE} g_R(x) \dd x = q^{d-1}\per(E) - q^d \lambda \int_E g_{R/q}(y) \dd y,\]
so the minimization problem for these rescaled sets is equivalent to the original one with $\lambda$ replaced by $q \lambda$ and $R$ replaced by $R/q$.
\end{remark}

The choice of signs in the notation is motivated by \eqref{eq:kappaplus} and \eqref{eq:kappaminus}, and by the fact that the free boundaries of $D^\lambda$ and $D^{-\lambda}$ have curvature $\lambda$ and $-\lambda$ respectively with respect to their outer normals, see \eqref{eq:intbypartsH}. 

\begin{remark}
From now on, whenever we use the variational mean curvature $\kappa_D$ for some $D \subset B(0,1)$, we will always assume that the density used is $g_2$, as in Definition \ref{def:dlambda} above.
\end{remark}

We will see in later sections that for large values of $\lambda$, the sets $D^\lambda$ and $D^{-\lambda}$ provide us with an approximation of $D$ in Hausdorff distance from the inside and outside respectively, motivating the notation. Moreover, they also determine the curvature $\kappa_D$ through \eqref{eq:kappaplus} and \eqref{eq:kappaminus}.

\subsection{Bounds and examples of variational mean curvatures}

\begin{lemma}\label{lem:ballinc}Assume that $x_0, r$ are such that $\Br \subseteq D$ up to measure zero, that is, $|\Br \setminus D|=0$. Then the optimal variational mean curvature $\kappa_D$ of $D$ satisfies 
\begin{equation}\label{eq:curvinball}\restr{\kappa_D}{\Br} \ls \frac{d}{r}.\end{equation}
In consequence, for any interior point $x \in \accentset{\circ}{D}$, we have
\begin{equation}\label{eq:distboundin}\kappa_D(x) \ls \frac{d}{\dist(x, \partial D)}.\end{equation}
Similarly, for $x \in \R^d \setminus \overline D$ we have $-\kappa_D(x) \ls d/\dist(x, \partial D)$. Therefore, for any $K \subset \R^d$ we have 
\begin{equation}\label{eq:distbound}\|\kappa_D\|_{L^\infty(K)} \ls \frac{d}{\dist(K, \partial D)},\end{equation}
where $\dist(K, \partial D) := \inf_{x \in K} \dist(x, \partial D)$.
\end{lemma}
\begin{proof}
By the definition of $D^\lambda$ as the maximal solution of \eqref{eq:Dlambdaplusprob1} and that of $\kappa_D$ in \eqref{eq:kappaplus}, $x \in D^\lambda$ implies $\kappa_D(x) \ls \lambda$. On the other hand, by Lemma \ref{lem:basiccomp} and since by Example \ref{ex:curvball} we know that we can find a variational mean curvature for $\Br$ with value $d/r$ in $\Br$, we have that $\lambda > d/r$ implies $\Br \subseteq D^\lambda$, so for $x \in \Br$ we get $\kappa_D(x)\ls \lambda$ for every $\lambda > d/r$, which is \eqref{eq:curvinball}. To see \eqref{eq:distboundin}, just notice that since $x \in \accentset{\circ}{D}$, we have that $B(x,r) \subset D$ for each $r < \dist(x, \partial D)$.

If $x \in \R^d \setminus \overline D$, we can proceed similarly using $F^\lambda = \R^d \setminus D^{-\lambda}$ and its variational problem \eqref{eq:outsideprob}. These two cases prove \eqref{eq:distbound}, since the bound is trivial when $\dist(K,\partial D)=0$.
\end{proof}

\begin{lemma}\label{lem:cheeger}
Let $D \subset \R^d$ be bounded. Denote by
\[h(D):= \min_{E \subset D}\frac{\per(E)}{|E|}\]
the Cheeger constant of $D$, the minimum being attained at Cheeger sets of $D$. Then $\kappa_D(x) \gs h(D)$ for $x \in D$, with equality for $x \in C_D$, the maximal Cheeger set.
\end{lemma}
\begin{proof}
We again consider the problem
\begin{equation}\label{eq:lambdaprob}\min_{E \subseteq D} \per(E) - \lambda |E|,\end{equation}
with $D^\lambda$ its maximal solution. Assume $\lambda>0$ is such that $|D^\lambda| \neq 0$, then by comparing with the empty set we have
\[\per(D^\lambda) - \lambda |D^\lambda| \ls 0,\]
which implies
\[\lambda \gs \frac{\per(D^\lambda)}{|D^\lambda|}\gs \min_{E \subset D}\frac{\per(E)}{|E|}=h(D),\]
which implies, recalling \eqref{eq:kappaplus}, that $\kappa_D(x)\gs h(D)$ for all $x \in D$. Similarly, considering \eqref{eq:lambdaprob} for $\lambda = h(D)$ we get
\[\per(D^{h(D)})-h(D)|D^{h(D)}| \ls 0,\]
so that $D^{h(D)}$ is a Cheeger set, and in fact by maximality $D^{h(D)} = C_D$, which in turn implies $\kappa_D(x) = h(D)$ for $x \in C_D$.
\end{proof}

\begin{proposition}\label{prop:curvsquare}
Let $S=(0,1)\times (0,1)$ be the unit square in $\R^2$. Then denoting by $Q^1:=(0,1/2)\times(0,1/2)$ the lower left quadrant we have
\begin{equation}\label{eq:curvsquare} \kappa_S(x) = \begin{cases}
h(S)=1/r_S:=2+\sqrt{\pi} & \text{ if } x \in C_S,\\
(x_1+x_2+\sqrt{2 x_1 x_2})^{-1} & \text{ if } x \in (S \setminus C_S)\cap Q^1,
\end{cases}\end{equation}
and similarly for the other quadrants. The Cheeger set $C_S$ (unique by convexity, see \cite{AltCas09}) is given by
\begin{equation*}\label{eq:cheegersquare}
C_S = \left\{ x \in S \, \middle \vert \, \dist\!\big(x, (r_S, 1-r_S)\times (r_S, 1-r_S) \big) < r_S\right\}.
\end{equation*}
\end{proposition}
\begin{proof}If $x \in C_S$ we use Lemma \ref{lem:cheeger}; the value of $r_S$ can be found for example in \cite[Thm.~3]{KawLac06}. If $x \notin C_S$, without loss of generality we can assume that $x=(x_1,x_2)$ is in $(S \setminus C_S)\cap Q^1$. Now, $x$ belongs to the circle centered at $(R(x),R(x))$ with radius $R(x)$, that is
\[(x_1 - R(x))^2+(x_2 - R(x))^2=R(x)^2.\]
This is a quadratic equation for $R(x)$ that we can solve to find $\kappa_S(x)=1/R(x)$.
\end{proof}

\begin{remark}\label{rem:offsets}
It is proved in \cite[Thm.~3.32(i)]{StrZie97} that for a convex planar set $E$ and $\lambda$ small enough, $E^\lambda$ can be written as a union of balls of radius $1/\lambda$. In this situation, the proof of \cite[Thm.~1]{KawLac06} building up on this characterization as a union implies in particular that 
\[E^\lambda = [E]_i^{1/\lambda} + \frac{1}{\lambda}B(0,1),\]
where $[E]_i^{1/\lambda} \subset E$ is the $1/\lambda$-offset of $E$ in the direction of its inner normal, so that $E^\lambda$ is obtained by a ``rolling ball'' procedure. Furthermore, it has been recently proven \cite{LeoSar19} that the same characterization holds for more general planar sets, namely Jordan domains with no necks.
\end{remark}

\section{Convergence for indicatrices with noise}\label{sec:local}
In this section we prove Theorem \ref{thm:localhausdorff} on Hausdorff convergence of level sets for denoising the indicatrix of a bounded finite perimeter set $D \subset B(0,1)$, with the variational mean curvature $\kappa_D \in L^1(\R^d)$ constructed with the choices of Definition \ref{def:dlambda} used in the formulas \eqref{eq:kappaplus} and \eqref{eq:kappaminus}, and with noise and parameter choice controlled by \eqref{eq:hardparamchoice}. To this end, let $u_{\alpha,0}$ be the precise representative of the minimizer of \eqref{eq:ROFpsi} with $f=1_D$ and $w=0$, that is, with no noise added. We denote by 
\[\kappa_\alpha := v_{\alpha, 0} = \frac{1}{\alpha}\psi'(1_D - u_{\alpha,0})\]
the corresponding variational mean curvature associated to the level sets of $u_{\alpha,0}$ through duality in Proposition \ref{prop:dualROFpsi}. The definition of $\kappa_D$ also provides us with a natural precise representative for it; we will implicitly use these precise representatives in the rest of the section.

Now, using Corollary \ref{cor:nonewlevels} and since $u_{\alpha,0}$ is the result of denoising $1_D$ which has values in $[0,1]$, we only need to consider $s \in (0,1)$ and consequently $E_\alpha^s = \{ u_{\alpha, 0} >s\}$ always denotes upper level sets (c.f.~Proposition \ref{prop:slicedROFpsi}).

We implement the local strategy of \cite[Thm.~2]{ChaDuvPeyPoo17}, which requires that $v_{\alpha,0}$ is $L^d$-equiintegrable only on $K_\delta = \{\dist(\cdot, \partial D)\gs\delta\}$ for each $\delta >0$. This equiintegrability is in turn a consequence of Lemma \ref{lem:ballinc} and Proposition \ref{prop:curvgodown} below, which combine to give the bound 
\begin{equation}\label{eq:curvbound}\|v_{\alpha, 0}\|_{L^\infty(K_\delta)}\ls \frac{d}{\delta}.\end{equation}

\begin{proposition}
\label{prop:curvgodown}
The noiseless dual variable $\kappa_\alpha$ satisfies $|\kappa_\alpha| \ls |\kappa_D|$ almost everywhere.
\end{proposition}

To prove this proposition, which depends crucially on Assumption \eqref{eq:psiassum}, we will use the following lemmas:
\begin{lemma}\label{lem:convsupport}
 The denoised solution $u_{\alpha,0}$ with $f=1_D$ and $w=0$ satisfies \[u_{\alpha,0}(x) = 0 \text{ for a.e. }x \in \R^d \setminus \conv D.\]
\end{lemma}
\begin{proof}
Denote $u:=u_{\alpha,0}$, and assume for the sake of contradiction that \[|\{ u \neq 0\}\setminus \conv D|>0.\] This implies, defining $u_c:= u 1_{\conv D}$, that
\begin{align*}\int_{\R^d} \psi(u-1_D) &= \int_{\conv D} \psi(u-1_D) + \int_{\R^d \setminus \conv D} \psi(u) \\&> \int_{\conv D} \psi(u-1_D) = \int_{\R^d} \psi(u_c-1_D).\end{align*}
Moreover we can write the coarea formula for $u$ as
\begin{align*}\TV(u)&=\int_0^{+\infty} \per(\{u>s\}) \dd s + \int_{-\infty}^0 \per(\{u<s\}) \dd s\\ 
&\gs \int_0^{+\infty} \per(\{u>s\}\cap \conv D) \dd s + \int_{-\infty}^0 \per(\{u<s\} \cap \conv D) \dd s \\
& = \int_0^{+\infty} \per(\{u_c>s\}) \dd s + \int_{-\infty}^0 \per(\{u_c<s\}) \dd s = \TV(u_c),
\end{align*}
where we have used that the level sets $\{u>s\}$ for $s>0$ and $\{u<s\}$ for $s<0$ must have finite mass since $u \in L^{d/(d-1)}$, and the convexity of $\conv D$. These two inequalities mean that $u$ could not be a minimizer.
\end{proof}

\begin{lemma}
\label{lem:inclusions}
Let $0 < s \ls 1$. Then 
\begin{equation}\label{eq:DlambsubsE}D^\lambda \subset E_\alpha^s \text{ when } 0 < \lambda < \psi'(1-s)/\alpha,\end{equation} whereas 
\begin{equation}\label{eq:EsubsDlamb}E_\alpha^s \subset D^{-\lambda} \text{ for }0 > -\lambda > -\psi'(s)/\alpha.\end{equation}
\end{lemma}
\begin{proof}
First, notice that the problems \eqref{eq:insideprob} and \eqref{eq:outsideprob} satisfied by $D^\lambda$ and $D^{-\lambda}$ are of obstacle type, so that as in \cite[Lem.~9]{IglMerSch18}, one can lift the obstacle constraint and conclude that $D^\lambda$ minimizes
\begin{equation}\label{eq:kappailambda} E \mapsto \per(E) - \int_E \kappa_i^\lambda \quad \text{with} \quad \kappa_i^\lambda = \lambda 1_{D} + \kappa_D 1_{\R^d \setminus D}\end{equation}
whereas $D^{-\lambda}$ minimizes
\[ E \mapsto \per(E) - \int_E \kappa_o^\lambda \quad \text{with} \quad \kappa_o^\lambda =-\lambda g_2 1_{\R^d \setminus D} - \kappa_D 1_{D}.\]
Therefore, we can write
\[\per(E_\alpha^s \cap D^\lambda) - \int_{E_\alpha^s \cap D^\lambda} \kappa_i^\lambda \gs \per(D^\lambda) - \int_{D^\lambda} \kappa_i^\lambda.\]
On the other hand, the level-set $E_\alpha^s$ has a curvature $\kappa_\alpha$, which allows writing
\[ \per(E_\alpha^s \cup D^\lambda) - \int_{E_\alpha^s \cup D^\lambda} \kappa_\alpha \gs \per(E_\alpha^s) - \int_{E_\alpha^s} \kappa_\alpha.\]
Summing these two inequalities, we obtain
\[\int_{D^\lambda \setminus E_\alpha^s} \kappa_i^\lambda \gs \int_{D^\lambda \setminus E_\alpha^s} \kappa_\alpha \]
that rewrites, since $D^\lambda \subset D$ and using the definition of $\kappa_i^\lambda$ in \eqref{eq:kappailambda}, 
\begin{equation} \int_{D^\lambda \setminus E_\alpha^s} \left( \lambda - \frac{\psi' (1-u_\alpha)}{\alpha} \right) \gs 0. \label{eq:contr} \end{equation} 
Now, on $(E_\alpha^s)^c$ by definition $u_\alpha \ls s$ which, since $\psi'$ is strictly increasing, implies $\lambda - \psi'(1-u_\alpha)/\alpha \ls \lambda - \psi'(1-s)/\alpha$.
Therefore, for $\lambda < \psi'(1-s)/\alpha$, \eqref{eq:contr} can hold only if $|D^\lambda \setminus E_\alpha^s| = 0$, that is $D^\lambda \subset E_\alpha^s$ a.e.~
 
Similarly, we can write
\[\per(E_\alpha^s \cup D^{-\lambda}) - \int_{E_\alpha^s \cup D^{-\lambda}} \kappa_o^\lambda \gs \per(D^{-\lambda}) - \int_{D^{-\lambda}} \kappa_o^\lambda,\]
\[\per(E_\alpha^s \cap D^{-\lambda}) - \int_{E_\alpha^s \cap D^{-\lambda}} \kappa_\alpha \gs\per(E_\alpha^s) - \int_{E_\alpha^s} \kappa_\alpha,\]
to sum these inequalities and, using $(D^{-\lambda})^c \subset D^c$ and $u_\alpha > s$ on $E^s_\alpha$, obtain
\[ 0 \gs \int_{E_\alpha^s \setminus D^{-\lambda}} \left(\frac{\psi'(u_\alpha)}{\alpha} - \lambda g_2 \right) > \int_{E_\alpha^s \setminus D^{-\lambda}} \left(\frac{\psi'(s)}{\alpha} - \lambda g_2 \right).\]
Hence, since we have $g_2 \ls 1$, as soon as $0 < \lambda < \psi'(s)/\alpha$, we obtain $E_\alpha^s \subset D^{-\lambda}$ a.e.~
\end{proof}

\begin{proof}[Proof of Proposition \ref{prop:curvgodown}] First we take $x \in D$ (this implies by definition $\kappa_D(x) \gs 0$) and define $s:=u_{\alpha, 0}(x)$, implying $\kappa_\alpha(x) = \psi'(1-s)/\alpha$f, and assume that for some $\eps>0$
\begin{equation}\label{eq:poswrongassum}\kappa_\alpha(x) = \frac{\psi'(1-s)}{\alpha}\gs \kappa_D(x)+\eps,\end{equation} to then use Lemma \ref{lem:inclusions} to derive a contradiction. By definition of the level sets we have that for all $\delta > 0$,
\begin{equation}\label{eq:defls}x \in E_\alpha^{s-\delta}, \text{ and } x \notin E_\alpha^{s+\delta}.\end{equation}
This, combined with \eqref{eq:DlambsubsE} implies that $x \notin D^\lambda$ whenever $0< \lambda < \psi'(1-s-\delta)/\alpha$. On the other hand, since $g_2(x)=1$ the construction \eqref{eq:kappaplus} of $\kappa_D$, \eqref{eq:poswrongassum} and \eqref{eq:EsubsDlamb} give $x \in D^\lambda$ for all $\lambda \gs \psi'(1-s)/\alpha-\eps> 0$, where for the last inequality we have used \eqref{eq:poswrongassum}. Choosing $\delta$ such that \[\psi'(1-s)-\psi'(1-s-\delta) \ls \alpha \eps,\] which is possible since $\psi \in \mathcal{C}^1(\R)$, these two statements are contradictory and therefore we must have $\kappa_\alpha(x) \ls \kappa_D(x)$ for all $x \in D$.

Now, if $x \in \R^d \setminus \overline D$, by Lemma \ref{lem:convsupport} we can assume $x \in \conv D \setminus \overline D$, since otherwise we would have $\kappa_\alpha(x) = 0$ and the inequality is trivially satisfied. This implies in particular that $g_2(x)=1$ (see \eqref{eq:defgR})  and we can proceed similarly as above. We have $\kappa_\alpha(x) = \psi'(-s)/\alpha$ and $\kappa_D(x) \ls 0$, and failure of the statement means that for some $\eps > 0$ we have
\begin{equation*}\label{eq:negwrongassum}\kappa_\alpha(x)=\frac{\psi'(-s)}{\alpha}\ls \kappa_D(x)-\eps.\end{equation*}
As before, for any $\delta$ \eqref{eq:defls} holds and \eqref{eq:EsubsDlamb} then implies that $x \in D^{-\lambda}$ as soon as $0> -\lambda > \psi'(-s+\delta)/\alpha$. However the definition of $\kappa_D$ \eqref{eq:kappaminus} and that $g_2(x)=1$ imply that $x \notin D^{-\lambda}$ if $-\lambda < \kappa_D(x) \ls \psi'(-s)/\alpha+ \eps < 0$, so that if $\delta$ is such that $\psi'(-s+\delta)-\psi'(-s) \ls \alpha \eps$ we again derive a contradiction. 
\end{proof}

\begin{remark}It might seem slightly surprising that even though the construction of $\kappa_D$ depends on the density $g$, as has been seen in Proposition \ref{prop:emptyFlambda}, we can still obtain the inequality $|\kappa_\alpha| \ls |\kappa_D|$. There are two reasons for this. First, we were able to bound the support of $u_\alpha$ in Lemma \ref{lem:convsupport}, allowing to avoid the unintuitive behaviour of $\kappa_D$ for small negative values and far away from $D$. Second, with our particular choice \eqref{eq:defgR} we have $g_2 = 1$ in $\conv D$, so we can still obtain the desired comparison without distorting the values of $\kappa_\alpha$ in question.
\end{remark}

As a consequence of \eqref{eq:curvbound} we can obtain uniform density estimates for the level-sets $E^s_{\alpha, w} := \{u_{\alpha,w} >s\}$ outside of $K_\delta$ with scale $r_\delta$ and constant $C_\delta$ possibly degenerating as $\delta \to 0$. These are proved in Proposition \ref{prop:densityest} of Appendix \ref{sec:densityest}. A further consequence of these density estimates is the following compact support result.

\begin{proposition}\label{prop:compsupp}Assume a parameter choice such that 
\begin{equation}\label{eq:isoperstabil}\|v_{\alpha, w} - v_{\alpha, 0}\|_{L^d(\R^d)} \ls C_0 < \Theta_d,\end{equation}
where $v_{\alpha,w}$ is the duality certificate for \eqref{eq:ROFpsi} of Proposition \ref{prop:dualROFpsi}, and that $f = 1_D$ for $D$ bounded. Then there is $R>0$ such that
\begin{equation*}\label{eq:compsupp}\supp u_{\alpha, w} \subset B(0,R),\end{equation*}
with $R$ depending on $C_0$ but not on the specific $\alpha$ and $w$.
\end{proposition}
\begin{proof}
Denote $E:=E^s_{\alpha, w} = \{u_{\alpha,w} >s\}$ for $s>0$ and $\{u_{\alpha,w} <s\}$ if $s<0$ where $\alpha, w$ are fixed. Since $\per(E)=\sign(s)\int_E v_{\alpha, w}$ by Proposition \ref{prop:slicedROFpsi}, using the H{\"o}lder inequality, \eqref{eq:isoperstabil}, the isoperimetric inequality \eqref{eq:isoper} and Proposition \ref{prop:curvgodown} -- recalling that $v_{\alpha,0} = \kappa_\alpha$, we get
\begin{equation}\label{eq:perbound}\begin{aligned}
\per(E) &\ls \left \vert \int_E (v_{\alpha,w} - v_{\alpha,0}) \right \vert + \left \vert \int_E v_{\alpha,0} \right \vert \\ 
&\ls C_0|E|^{(d-1)/d} + \int_E \vert v_{\alpha,0} \vert \ls \Theta_d^{-1} C_0 \per(E) + \|\kappa_D\|_{L^1},
\end{aligned}\end{equation}
which provides a uniform bound for $\per(E)$.

To prove \eqref{eq:compsupp} we derive a uniform bound for $\mathrm{diam}(E)$, which does not follow from the perimeter bound alone. However, we have that by \eqref{eq:curvbound} the hypotheses of Proposition \ref{prop:densityest} are satisfied with $K=\{\dist(\cdot, \partial D)\gs 1\}$ and \[r_{K,\eps} = \frac{\eps^{1/d}}{d|B(0,1)|^{1/d}},\]
so the $E$ satisfy uniform density estimates at some scale $r_K$ and with constant $C_K$ outside the bounded set $K$, which are then combined with the fact that \eqref{eq:perbound} and the isoperimetric inequality imply a uniform upper bound for $|E|$. We can then prove the diameter bound by considering points $x_n \in \partial E \setminus K$, $n=1,\ldots,N$ with $|x_i - x_j| > r_0$ for $i \neq j$ and some $r_0 \in (0, r_K)$. The inner density estimate $|E \cap B(x_n, r_{0})| > C|B(x_n, r_0)|$ combined with the uniform bound for $|E|$ gives a uniform upper bound for $N$, hence also for $\mathrm{diam}(E)$.
\end{proof}

We are now ready to prove the main result of this section.

\subsection{Proof of Theorem \ref{thm:localhausdorff}}

\begin{proof}[Proof of Theorem \ref{thm:localhausdorff}]First, we notice that the definition of the Hausdorff distance reads
\[\dh(\partial E^s_{\alpha,w}, \partial D) = \max \left( \sup_{x \in \partial E^s_{\alpha,w}} \dist(x, \partial D), \sup_{x \in \partial D} \dist(x, \partial E^s_{\alpha,w})\right).\] 
Therefore, we need to prove the two statements
\begin{align}\label{eq:liminfhaus}\sup_{x \in \partial D} \dist(x, \partial E^s_{\alpha_n,w_n}) &\to 0, \text{ and }\\
\label{eq:limsuphaus}\sup_{x \in \partial E^s_{\alpha_n,w_n}} \dist(x, \partial D) &\to 0.\end{align}

Let us start with \eqref{eq:liminfhaus}, for which the argument follows closely that of \cite[Prop.~9]{ChaDuvPeyPoo17}, which we reproduce for completeness. 
Since the parameter choice \eqref{eq:hardparamchoice} implies in particular condition \eqref{eq:easyparamchoice}, arguing as in the proof of Proposition \ref{prop:existconv} we have, up to a subsequence, that $Du_{\alpha_n, w_n} \wksto D1_D$.
Now the coarea formula, as formulated for example in \cite[Thm.~10.3.3]{AttButMic14}, tells us that we can slice these measures (and not just their total variations) so that

\[Du_{\alpha_n, w_n}(A) = \int_{-\infty}^{+\infty} D 1_{\{u_{\alpha_n, w_n} > s\}}(A) \dd s = \int_{-\infty}^{+\infty} D 1_{E^s_{\alpha_n, w_n}}(A) \dd s \text{ for Borel sets }A,\]
and therefore for a.e.~$s \in (0,1)$ we in fact have
\begin{equation}\label{eq:wkstinds}D 1_{E^s_{\alpha_n, w_n}} \wksto D1_D.\end{equation}
Now let $x \in \supp D1_D$, then for any $r >0$ using \eqref{eq:wkstinds} we get
\[0 < |D1_D|(B(x,r)) \ls \liminf_n |D 1_{E^s_{\alpha_n, w_n}}|(B(x,r)),\]
which implies that $\limsup_n \dist(x, \supp D1_{E_{\alpha_n, w_n}}) \ls r$. Since $r>0$ was arbitrary we conclude $\dist(x, \supp D1_{E_{\alpha_n, w_n}}) \to 0$, and in particular
\[\sup_{x \in \supp D1_D} \dist(x, \supp D1_{E_{\alpha_n, w_n}}) \to 0.\]
Finally, as mentioned in Section \ref{subsec:preliminaries}, we always use representatives for $E_{\alpha_n, w_n}$ and $D$ for which
\[\partial E_{\alpha_n, w_n} = \supp D 1_{E_{\alpha_n, w_n}} \text{ and }\partial D = \supp D 1_D,\]
completing the proof of \eqref{eq:liminfhaus}.

To prove \eqref{eq:limsuphaus} we assume it does not hold to reach a contradiction. One the one hand, we notice that using the parameter choice \eqref{eq:hardparamchoice} and Proposition \ref{prop:dualstability},  we can apply Proposition \ref{prop:compsupp} to see that the $u_{\alpha_n,w_n}$ have a common compact support. This, combined with the convergence in $L^1_\loc$ of $u_{\alpha_n,w_n}$ proved in Proposition \ref{prop:existconv}, implies that $|E^s_{\alpha_n,w_n} \Delta D|\to 0$. On the other hand, if \eqref{eq:limsuphaus} fails, from Proposition \ref{prop:hausineq} and its proof we see that we must have either 
\begin{equation}\label{eq:insideoutside}\sup_{x \in E^s_{\alpha_n,w_n}} \dist(x, D) \not\to 0, \text{ or }\sup_{x \in \R^d \setminus E^s_{\alpha_n,w_n}} \dist\left(x, \R^d \setminus D\right)\not\to 0.\end{equation}
Assume the first is true, so that there is $\delta > 0$ and a subsequence $x_n \in E^s_{\alpha_n,w_n}$ for which $\dist(x_n, D) > \delta$. In that case, by the inner density estimates proved in Proposition \ref{prop:densityest} of Appendix \ref{sec:densityest}, we have that 
\begin{equation*}|E^s_{\alpha_n,w_n} \Delta D| \gs |E^s_{\alpha_n,w_n} \setminus D| \gs |B(x_n, \delta) \cap E^s_{\alpha_n,w_n}| \gs C_\delta |B(0,\delta)|,\end{equation*}
a contradiction. If the second case of \eqref{eq:insideoutside} was true, we instead use the outer density estimate to again contradict the $L^1$ convergence.
\end{proof}

\section{Convergence for generic {BV} functions without noise}\label{sec:noiseless}

This section is aimed at the proof of Theorem \ref{thm:generichausdorff} and simplified versions of it for piecewise constant functions. We are concerned with the noiseless situation, that is, we assume $w=0$ throughout. Moreover, we always assume $\supp f \subset B(0,1)$ which, arguing as in Lemma \ref{lem:convsupport}, implies
\begin{equation}\label{eq:boundedsupp}\supp u_{\alpha, 0} \subset \conv\left( \supp f \right) \subset \overline{B(0,1)}.\end{equation}

It might seem that also in this case the convergence could be obtained by following the same arguments used in Section \ref{sec:local}. This would require reproducing the inequality $|v_{\alpha,0}| \ls |\kappa_D|$ of Proposition \ref{prop:curvgodown}. However, it is not clear how to construct a function that acts as a variational curvature for all the level sets of $f$ simultaneously and also plays the role of $\kappa_D$ in the previous inequality, even for piecewise constant $f$. To gather some intuition, one can notice that when the data is a characteristic function $f=1_D$, the boundary of the noiseless level set $E_\alpha^s$ can be decomposed into two parts: a part of the boundary of $D$, and a free part (which is actually $C^{1,\gamma}$ for $\gamma\in (0,1/2)$ \cite[Thm.~21.8]{Mag12}) with mean curvature $(s-1)/\alpha$ or $-s/\alpha$ for points inside or outside of $D$ respectively. In turn, assuming $g=1$ at the points under consideration, the same is true for the sets $D^{\pm\lambda}$ defined with curvature $\pm \lambda$ and constraints to lie inside or outside of $D$, and used through \eqref{eq:kappaplus} and \eqref{eq:kappaminus} in the definition of $\kappa_D$. This allows to easily compare them to the mentioned level sets through their respective variational problems. Since at present we do not know which would be the variational problem defining a hypothetical analog of $D^{\pm\lambda}$, in this section we work with the same construction of curvatures of Section \ref{sec:curvatures}, but applied to each level set of $f$ separately.

\subsection{Approximation with the level sets of the optimal curvature of D}\label{sec:convind}
The key to the results of this section will be to know that we can approximate any $D \subset B(0,1)$ with the sets $D^\lambda$ and $D^{-\lambda}$ as $\lambda \to \infty$ arising from the choices of Definition \ref{def:dlambda}. First we note that this approximation happens in mass:
\begin{lemma}\label{lem:approxsetsmass}
For $D \subset \R^d$ bounded of finite perimeter, we have as $\lambda \to +\infty$ that
\[\left|D \setminus D^\lambda \right| \to 0, \text{ and }\left|D^{-\lambda} \setminus D \right| \to 0.\]
\end{lemma}
\begin{proof} It is contained in the proof of Proposition \ref{prop:kappawelldef}. 
For the inside approximants $D^\lambda$, the result is proven also in \cite[Thm.~2.3(ii)]{TamGia89}. 
\end{proof}
Moreover, this two-sided approximation also holds in Hausdorff distance of the corresponding boundaries:
\begin{lemma}\label{lem:approxsets}
For $D \subset \R^d$ bounded of finite perimeter and every $\eps >0$ there exists $\lambda_\eps >0$ such that $D^{\lambda_\eps} \subset D \subset D^{-\lambda_\eps}$, $\dh(\partial D^{\lambda_\eps} ,\partial D) \ls \eps$ and $\dh(\partial D^{-\lambda_\eps} ,\partial D) \ls \eps$.
\end{lemma}
\begin{proof}
The interior approximation is proved in \cite[Thm.~2.3(iv)]{TamGia89}; we reproduce their argument here, and see that it can also be applied for the exterior approximation with $D^{-\lambda}$.

To start, let $x \in D$ with $\dist(x, \partial D) > \eps$. Then we have that $B(x, \eps)\subset D$, and as in the proof of Lemma \ref{lem:ballinc} we must have $B(x, \eps) \subset D^\lambda$ for all $\lambda > \eps/d$, in particular $x \in D^\lambda \setminus \partial D^\lambda$, implying
\[\sup_{x \in \partial D^\lambda} \dist(x, \partial D) \ls \eps \text{ for all }\lambda > \eps/d.\]
For the other term of the Hausdorff distance, the strategy is to cover $\partial D$ with finitely many balls $B(x_j, \eps)$ with $j=1, \ldots, N_\eps$ and $x_j \in \partial D$, which is possible since $\partial D$ is bounded. Then, since Lemma \ref{lem:approxsetsmass} implies that $|D \setminus D^\lambda| \to 0$ as $\lambda \to \infty$, we can choose $\lambda_{\eps}$ such that $|D^{\lambda_\eps} \cap B(x_j, \eps)| >0$ for all $j$. Since these balls cover $\partial D$, we have that
\[\sup_{x \in \partial D} \dist(x, \partial D^\lambda) \ls \eps \text{ for all }\lambda > \lambda_\eps.\]
Since it was only used that $\partial D = \partial ( \R^d \setminus D )$ is bounded, we can proceed in the same way for the approximation with $\partial D^{-\lambda}$. For the first part, it suffices to notice that by definition $F^\lambda=\R^d \setminus D^{-\lambda}$ are minimizers of \eqref{eq:outsideprob}, so if $x \in \R^d \setminus D$ with $\dist(x , \partial D) > \eps$ we must also have $B(x, \eps) \subset F^\lambda = \R^d \setminus D^{-\lambda}$ for all $\lambda > \eps/d$. Moreover, we have $|D^{-\lambda} \setminus D|\to 0$ by Lemma \ref{lem:approxsetsmass}, which allows to repeat the covering argument.
\end{proof}

\begin{corollary}\label{cor:approxcomplements}
For every bounded finite perimeter set $D \subset \R^d$ and every $\eps >0$ there exists $\lambda_\eps >0$ such that
\[ \dh(D^{ \lambda_\eps} ,D) \ls \eps\text{ and }\dh(D^{-\lambda_\eps},D) \ls \eps,\]
and also 
\[ \dh(\R^d \setminus D^{\lambda_\eps} ,\R^d \setminus D) \ls \eps\text{ and }\dh(\R^d \setminus D^{-\lambda_\eps} ,\R^d \setminus D) \ls \eps.\]
\end{corollary}
\begin{proof}
It follows by Lemmas \ref{lem:approxsetsmass} and \ref{lem:approxsets} combined with Theorem \ref{thm:hausdorffequiv}. Note that the latter theorem is not quantitative and we could get different values of $\lambda_\eps$ from it for the different convergences, but we can then just use the largest of the two.
\end{proof}

\begin{example}
A result like Lemma \ref{lem:approxsets} can only hold for bounded sets. As a counterexample, consider $D$ defined by
\[D:= \bigcup_{j=0}^\infty B\left( (j, 0), \frac{1}{2^{j+1}}\right).\]
Clearly we have $|D| < \infty$ and $\per(D) < \infty$, but $D^\lambda$ must be a union of finitely many balls, so $\dh(\partial D, \partial D^\lambda) = \infty$ for all $\lambda > 0$.
\end{example}

To obtain Hausdorff convergence in the noiseless case, we do not need to use density estimates; for indicatrices it is enough to combine Lemmas \ref{lem:inclusions} and \ref{lem:approxsets}:
\begin{proposition}
Let $f=1_D$ with $D \subset B(0,1)$, $w=0$, and denote by $u_{\alpha, 0}$ the coresponding minimizers of \eqref{eq:ROFpsi}. Then for almost every $s \in (0,1)$, the boundary $\partial E_\alpha^s$ of the level set $E_\alpha^s=\{u_{\alpha, 0} > s\}$ converges in Hausdorff distance to $\partial E_0^s = \partial \{f > s\}$ as $\alpha \to 0$.
 \label{prop:noiselesshaus}
\end{proposition}

\subsection{Denoising of a generic BV function. Proof of Theorem \ref{thm:generichausdorff}}
Let now $f$ be a generic $\BV$ function supported in $B(0,1)$ and let $u_\alpha$ the minimizer of \eqref{eq:ROFpsi} with $w=0$. One wants to reproduce the construction of Section \ref{sec:convind} for every level set of $u_\alpha$. We denote $G^s := \{f > s\}$ for $s>0$ and $G^s := \{f < s\}$ for $s<0$ the level sets of $f$ and similarly $E_\alpha^s = \{u_\alpha > s\}$ for $s>0$, $E_\alpha^s = \{u_\alpha < s\}$ for $s<0$ the ones of $u_\alpha$. The sets $E_\alpha^s$ minimize
\begin{equation} \label{eq:ROFls} E \mapsto \per(E) + \frac{\sign(s)}{\alpha} \int_E \psi'(s-f).\end{equation}
In contrast to the situation in Section \ref{sec:local}, the functions involved may take negative values, but by using lower level sets for $s<0$ we ensure that these are also contained in $B(0,1)$. With this in view, given $\eps>0$ we have by Lemma \ref{lem:approxsets} for each $s$ an approximation $(G^s)^{\lambda(\eps,s)}$ of $G^s$ from inside, which to make the notation slightly lighter we denote as 
\[G_{i,\eps}^s:=(G^s)^{\lambda(\eps,s)} \subset G^s,\]
with $\dh(\partial G_{o,\eps}^s, \partial G^s) < \eps$ and curvature $\kappa^s_{i,\eps}$ bounded above by $\lambda(\eps,s)$. Similarly we denote by 
\[G_{o,\eps}^s := (G^s)^{-\lambda(\eps,s)} \supset G^s\]
the corresponding approximation of $G^s$ from the outside with $\dh(\partial G_{o,\eps}^s, \partial G^s) < \eps$ and curvature $\kappa^s_{o,\eps}$ bounded below by $-\lambda(\eps,s)$ on $B(0,1)$.

\begin{lemma}
 \label{lem:delta}
Let $\eps, \delta >0$. Then for $\alpha$ small enough (depending on $s$, $\delta$ and $\eps$), 
\[|G_{i,\eps}^{s+\delta} \setminus E_\alpha^{s}| = 0\text{ and }|E_\alpha^{s} \setminus G_{o,\eps}^{s-\delta}|=0 \text{ for }s>0,\]
and analogously
\[|G_{i,\eps}^{s-\delta} \setminus E_\alpha^{s}| = 0\text{ and }|E_\alpha^{s} \setminus G_{o,\eps}^{s+\delta}|=0 \text{ for }s<0.\]
\end{lemma}
\begin{proof}
We assume that $s>0$, since the case $s<0$ follows in a completely analogous way after noticing that using lower level sets induces a change of sign in \eqref{eq:ROFls}, as well as a change in the direction of inclusions with respect to $s$.
Therefore, let $s>0$ and $\delta >0$ be fixed. To simplify notation in the proof, since $\eps$ is also fixed, we omit it and denote
\begin{gather*}G_i^{s+\delta}:=G_{i,\eps}^{s+\delta}=(G^{s+\delta})^{\lambda(\eps,{s+\delta})}, \,\ \kappa^{s+\delta}_i:=\kappa^{s+\delta}_{i,\eps} \ls \lambda(\eps,{s+\delta}), \\ G^{s-\delta}_o:=G^{s-\delta}_{o,\eps} = (G^{s-\delta})^{-\lambda(\eps,s-\delta)} \,\text{ and }\,\kappa^{s-\delta}_o:=\kappa^{s-\delta}_{o,\eps} \gs -\lambda(\eps,s-\delta).\end{gather*}
Using the minimality of $\eas$ in \eqref{eq:ROFls}, one can write
\[ \per(\eas) + \int_{\eas} \frac{\psi'(s-f)}{\alpha} \ls \per(\eas \cup G_i^{s+\delta}) + \int_{\eas \cup G_i^{s+\delta}} \frac{\psi'(s-f)}{\alpha}.\]
On the other hand, since $\kappa^{s+\delta}_i$  is a variational curvature of $G_i^{s+\delta}$, one has
\[ \per(G_i^{s+\delta}) - \int_{G_i^{s+\delta}} \kappa^{s+\delta}_i \ls \per(G_i^{s+\delta} \cap \eas) - \int_{G_i^{s+\delta} \cap \eas} \kappa^{s+\delta}_i.\]
Summing these two inequalities and using \eqref{eq:subaddper}, we get
\begin{equation}-\int_{G_i^{s+\delta} \setminus \eas} \kappa^{s+\delta}_i \ls  \int_{G_i^{s+\delta} \setminus \eas} \frac{\psi'(s-f)}{\alpha}. \label{eq:ineq}\end{equation}
Now, recall that $G_i^{s+\delta} \subset G^{s+\delta}$, meaning that $f \gs s+\delta$ on this set. Hence $s-f \ls -\delta$ and, since $\psi'$ is increasing and $\psi$ is even, \eqref{eq:ineq} implies
\[\int_{G_i^{s+\delta} \setminus \eas}  \left( -\kappa^{s+\delta}_i + \frac{\psi'(\delta)}{\alpha} \right) \ls 0. \]
Since $\kappa^{s+\delta}_i \ls \lambda(\eps,s+\delta)$, as soon as $\alpha \ls \psi'(\delta) / \lambda(\eps,s+\delta)$, which is always possible since $\psi'(\delta)>0$ by strict monotonicity, one must have $|G_i^{s+\delta} \setminus \eas| = 0.$

Similarly, the equality $|E_\alpha^{s} \setminus G_o^{s-\delta}|=0$ is obtained writing
\[ \per(\eas) + \int_{\eas} \frac{\psi'(s-f)}{\alpha} \ls \per(\eas \cap G_o^{s-\delta}) + \int_{\eas \cap G_o^{s-\delta}} \frac{\psi'(s-f)}{\alpha}\]
and
\[\per(G_o^{s-\delta}) - \int_{G_o^{s-\delta}} \kappa^{s-\delta}_o \ls \per(G_o^{s-\delta} \cup \eas) - \int_{G_o^{s-\delta} \cup \eas} \kappa^{s-\delta}_o.\]
Summing these inequalities we obtain
\[ \int_{\eas \setminus G_o^{s-\delta}} \frac{\psi'(s-f)}{\alpha} \ls - \int_{\eas \setminus G_o^{s-\delta}} \kappa^{s-\delta}_o.\]
Now, the complement of $G_o^{s-\delta}$ contains the complement of $G^{s-\delta}$, therefore on this set, one has $f \ls s-\delta$, which implies
\[ \int_{\eas \setminus G_o^{s-\delta}} \left( \frac{\psi'(\delta)}{\alpha} + \kappa^{s-\delta}_o \right) \ls 0,\]
which, together with $\kappa^{s-\delta}_o \gs - \lambda(\eps,s-\delta)$ on $B(0,1)$ and \eqref{eq:boundedsupp} forces the expected equality as soon as $\alpha \ls  \psi'(\delta)/\lambda(\eps,s-\delta)$.
\end{proof}

We can now prove the main result of this section:

\begin{proof}[Proof of Theorem \ref{thm:generichausdorff}]
The proof strongly relies on Lemma \ref{lem:delta}, and once again we assume without loss of generality that $s>0$ is fixed. Let $\eta >0$ and $\eps = \eta/2$. First, we show that one can find $\alpha_0$ such that $\dh(E_\alpha^s, G^s) \ls \eta$ for every $\alpha \ls \alpha_0.$
Using assumption \eqref{eq:assumhausdorff}, there exists $\delta >0$ such that  
 \begin{equation}\dh(G^{s\pm \delta},G^s) \ls \eps. \label{eq:choicedelta} \end{equation} Then, Lemma \ref{lem:delta} ensures the existence of $\alpha_0$ such that for $\alpha \ls \alpha_0$, we have up to measure zero $G_{i,\eps}^{s+\delta} \subset E_\alpha^s \subset G_{o,\eps}^{s-\delta}$. Now, we just have to note that 
 \[\dh(E_\alpha^s,G^s) = \max \left\{ \sup_{x \in \eas} \dist(x,G^s), \, \sup_{x \in G^s} \dist(x,\eas) \right \},\]
for which
\[G_{i,\eps}^{s+\delta} \subset E_\alpha^s \Rightarrow \sup_{x \in G^s} \dist(x,\eas)  \ls \sup_{x \in G^s} \dist(x,G_{i,\eps}^{s+\delta})\] and \[E_\alpha^s \subset G_{o,\eps}^{s-\delta} \Rightarrow \sup_{x \in \eas} \dist(x,G^s) \ls \sup_{x\in G_{o,\eps}^{s-\delta}} \dist(x,G^s).\]
Now, the triangle inequality for the Hausdorff distance, \eqref{eq:choicedelta}, the Hausdorff convergence of Lemma \ref{lem:approxsets} and Theorem \ref{thm:hausdorffequiv} imply
\begin{align*}\sup_{x \in G^s} \dist(x,G_{i,\eps}^{s+\delta}) &\ls \dh(G^s,G_{i,\eps}^{s+\delta}) \\& \ls \dh(G^s, G^{s+\delta}) + \dh(G^{s+\delta}, G_{i,\eps}^{s+\delta}) \ls 2\eps = \eta,\end{align*}
and
\begin{align*}\sup_{x \in G_{o,\eps}^{s-\delta}} \dist(x,G^s) &\ls \dh(G^s,G_{o,\eps}^{s-\delta}) \\&\ls  \dh(G^s, G^{s-\delta}) + \dh(G^{s-\delta}, G_{o,\eps}^{s-\delta}) \ls 2\eps = \eta.\end{align*}
We therefore conclude \[ \dh(G^s, E_\alpha^s) \ls 2\eps = \eta \text{ for }\alpha \ls \alpha_0,\] as claimed.

Similarly, using the second part of \eqref{eq:assumhausdorff} we notice that Lemma \ref{lem:delta} also provides us with the reverse inclusions for the complements
\[\R^d \setminus G_{i,\eps}^{s+\delta} \,\supset\, \R^d \setminus E_\alpha^s \,\supset\, \R^d \setminus G_{o,\eps}^{s-\delta},\]
so we find, possibly reducing $\alpha_0$, that also 
\[\dh(\R^d \setminus E_\alpha^s, \R^d \setminus G^s) \ls \eta \text{ for }\alpha \ls \alpha_0.\]
Using inequality \eqref{eq:hausineq} of Proposition \ref{prop:hausineq}, convergence in Hausdorff distance of the sets $E^s_\alpha$ and their complements as $\alpha \to 0$ implies convergence of the boundaries.
\end{proof}

\begin{remark}
We recall that $|E_\alpha^s \Delta G^s| \to 0$ holds for a.e.~$s$ because of the strong $L^1$ convergence $u_\alpha \to u$, which is implied by the support bound \eqref{eq:boundedsupp} and the $L^1_\loc$ convergence proved in Proposition \ref{prop:existconv}.
\end{remark}

The assumption we used for the level sets $G^s$ holds for many well-behaved functions, in particular:

\begin{proposition}\label{prop:densityimplies}
Let $f$ be such that the level sets $G^s$ satisfy uniform density estimates at some scale $r_0$ and constant $C$, independent of the level $s$. Then \eqref{eq:assumhausdorff} holds for a.e.~$s$.
\end{proposition}
\begin{proof}
By the assumption and arguing as for \eqref{eq:intinnerdens} in Proposition \ref{prop:hausdorff_of_sets}, we have that at any point $x \in G^s$ we have the inner density estimate
\begin{equation}\label{eq:innerdens}\frac{|G^s \cap B(x,r)|}{|B(x,r)|} \gs \overline{C},\end{equation}
for $r \ls \overline{r}_0 = 2r_0$ and $r_0, \overline{C} = C/2^d$ independent of $x$ and $s$. Moreover, since the $G^s$ are decreasing in $s$ we may assume $\delta > 0$ when taking the limit, and to conclude that $\dh(G^s, G^{s+\delta}) \to 0$ we just need to check
\begin{equation}\label{eq:halfhaus}\sup_{x \in G^{s}}\dist(x, G^{s+\delta}) \xrightarrow{\delta \to 0} 0,\end{equation}
since the other term in the Hausdorff distance vanishes. However, if \eqref{eq:halfhaus} were false, we can find $\{\delta_i\}_i$, $\rho >0$ and $x_{\delta_i} \in G^s$ such that $\dist(x_{\delta_i}, G^{s+\delta_i}) > \rho$. But using \eqref{eq:innerdens} for $G^s$ and $x_{\delta_i}$, and possibly reducing $\rho$ so that $\rho \ls \overline{r}_0$ that
\[|G^s \setminus G^{s+\delta_i}| \gs |G^s \cap B(x_{\delta_i},\rho)| \gs \overline{C} |B(x_{\delta_i},\rho)| = \overline{C} |B(0,\rho)|,\]
which is a contradiction with $|G^s \Delta G^{s+\delta}| \to 0$ as $\delta \to 0$ that clearly holds for a.e.~$s$.

Moreover, we also have the outer density estimate
\begin{equation}\label{eq:outerdens}\frac{|B(x,r) \setminus G^s|}{|B(x,r)|} \gs \overline{C},\end{equation}
again for $r \ls \overline{r}_0=2r_0$ and $\overline{C}=C/2^d$. Since the sets $\R^d \setminus G^s$ are increasing in $\delta$, to conclude that $\dh(\R^d \setminus G^{s}, \R^d \setminus G^{s+\delta}) \to 0$ we must check
\begin{equation*}\label{eq:anotherhalfhaus}\sup_{x \in \R^d \setminus G^{s+\delta}}\dist(x, \R^d \setminus G^s) \xrightarrow{\delta \to 0} 0.\end{equation*}
If this does not hold, we can find $\{\delta_i\}_i$, $\rho >0$ and $x_{\delta_i} \in \R^d \setminus G^{s+\delta_i}$ such that $\dist(x_{\delta_i}, \R^d \setminus G^s) > \rho$. But using \eqref{eq:outerdens} for $G^{s+\delta_i}$ and $x_{\delta_i}$ and with $\rho \ls \overline{r}_0$ we have
\begin{align*}\left|\left( \R^d \setminus G^{s+\delta_i}\right) \setminus \left( \R^d \setminus G^s \right) \right| = |G^s \setminus G^{s+\delta_i}| &\gs | B(x_{\delta_i},\rho) \setminus G^{s+\delta_i}| \\&\gs \overline{C} |B(x_{\delta_i},\rho)| = \overline{C} |B(0,\rho)|,\end{align*}
leading again to a contradiction.
\end{proof}

The results of \cite{ChaDuvPeyPoo17} or \cite{IglMer20} then directly imply that this assumption is also valid when the source condition holds:

\begin{corollary}\label{cor:sourceimplies}
Let $f$ be such that
\[\partial \TV(f) \neq \emptyset.\]
Then the level sets $G^s$ of $f$ satisfy \eqref{eq:assumhausdorff} for a.e.~$s$.
\end{corollary}

This conclusion is however nontrivial, since it could be that \eqref{eq:assumhausdorff} fails for a set of values of full measure:

\begin{example}\label{ex:pingpongballs}
Let $\{B_i\}_{i \gs 0}$ be a collection of balls such that 
\[B_i \subset B(0,1), \quad B_i \cap B_j = \emptyset \text{ if }i\neq j,\quad \sum_{i=0}^\infty \per(B_i) < +\infty \]
We construct a function 
\[f := \sum_{i = 0}^\infty a_i 1_{C_i} \text{, with }C_i = B_i + \sigma(i)\left(\frac32,0\right)\]
and values $a_i$ and offset signs $\sigma(i) \in \{-1, 1\}$ that we now describe. For that, let $\mathcal{B}$ be the subset of functions in $2^\mathbb{N}$ with finitely many nonzero values, and let us enumerate its elements as $\{b_i\}_{i \gs 0}$ in an order in which the position of their last nonzero value is increasing, say
\[0,\ 1,\ 01,\ 11,\ 001,\ 011,\ 101,\ 111,\ 0001\ldots\]
The $\sigma(i)$ are defined iteratively by 
\[\sigma(0)=1,\ \sigma(1)=-1, \text{ and }\sigma(i) = (-1)\big(\sigma \circ \iota \circ p\big)(b_i) \text{ for }i \gs 2,\]
where $\iota: \mathcal{B} \to \mathbb{N} \cup \{0\}$ gives the index in the enumeration described, and $p: \mathcal{B}\setminus \{0\} \to \mathcal{B}$ is the map that deletes the last nonzero element. Since $\iota \circ p\,(b_i)<i$, this process is well defined. Finally, let
\[a_i = \sum_{k=1}^\infty \frac{1}{2^k} b_i(k).\]
Now, for any value $s \in (0,1)$ with an infinite binary expansion and all irrationals in particular, when one denotes as $s_\ell$ the expansion of $s$ up to $\ell$ digits (so $s_\ell = a_{i_\ell}$ for some $i_\ell$) the corresponding offset signs $\sigma(i_\ell)$ alternate with $\ell$, so 
\[\dh\big(\{f \gs s_\ell\}, \{f \gs s\}\big) \gs 1\text{ while }s_\ell \to s\text{ as }\ell \to \infty.\]
\end{example}

\section{Can we have uniform density estimates at fixed scale?}\label{sec:uniform}
In Section \ref{sec:local} we have proved Hausdorff convergence of level sets for denoising of $1_D + w$ by using density estimates at scales that converge to $0$ as $\alpha \to 0$. However, as the next example shows, often more can be expected out of the denoised solutions:
\begin{example}\label{ex:squarecirc}
Consider for $\ell_n >r_n \to 0$, with $S=(0,1)^2 \subset \R^2$ and $\psi(t)=t^2/2$ the situation 
\[f=1_S, w_n=1_{B_n} \text{ with }B_n:=B\big((-\ell_n, -\ell_n), r_n\big), \text{ so that }\dist(S,B_n)=\ell_n - r_n.\]
The nontrivial level sets of $f+w_n$ are all $S \cup B_n$, and we clearly have that $\dh\big(\partial(S \cup B_n), \partial S\big) \to 0$, but they contain a spurious connected component not seen in the limit. Moreover we notice that if $\ell_n / r_n \to +\infty$ the sets $S \cup B_n$ fail to satisfy uniform density estimates, since in that case for any $x_n \in \partial B_n$ we have
\[\frac{\big|(S \cup B_n) \cap B(x_n, \ell_n - r_n)\big|}{|B(x_n, \ell_n - r_n)|}=\frac{|B_n|}{|B(x_n, \ell_n - r_n)|}\to 0.\]
Now, again using $\ell_n / r_n \to +\infty$ we have that $\per(\conv(S \cup B_n))>\per\big(S \cup B_n)$, so we have for the level sets of minimizers of \eqref{eq:ROFpsi} that $E^s_{\alpha_n, w_n} \subset S \cup B_n$. Moreover, if $s$ and $\alpha_n$ are such that $(1-s)/\alpha_n < 2/r_n$ we have that $E^s_{\alpha_n,w_n} \cap B_n = \emptyset$, as can be seen from the computations done in Example \ref{ex:curvball}. This implies, when using a linear parameter choice $\alpha_n = C \|w_n\|_{L^2} = C \sqrt{\pi} \,r_n$, that whenever $C >1/(2\sqrt{\pi})=1/\Theta_2$ the effect of $w_n$ is not seen in the solution. In that case it is easy to see that level sets admit uniform density estimates at fixed scale, since then $E^s_{\alpha_n, w_n}=S^{s/\alpha_n}$ where the notation $S^{s/\alpha_n}$ is understood in the sense of Definition \ref{def:dlambda}, which are explicitly computed for the case of the square $S$ in Proposition \ref{prop:curvsquare} and Remark \ref{rem:offsets}.
\end{example}
Uniform density estimates along the sequence of level sets of minimizers provide not only Hausdorff convergence of the boundaries of level sets, but also prevent the appearance of spurious structures smaller than a certain scale. For general sets of finite perimeter as in Section \ref{sec:local}, since the limit is not regular, we cannot in general expect uniform density for the level sets approaching it.

On the opposite side, if we knew that $\partial \TV(1_D) \neq \emptyset$, uniform density estimates for the level sets are implied by the results of \cite{ChaDuvPeyPoo17} for $d=2$ and \cite{IglMer20} for $d>2$. However, this source condition excludes large classes of sets $D$ where we would expect the level sets of minimizers to also satisfy uniform density estimates, in particular sets $D$ with general Lipschitz boundary and satisfying density estimates themselves, like the square in the example above. The question then arises of how to derive these estimates for the solutions in such cases. We are not able to give a complete answer but we collect some observations, specialized to the two-dimensional case and $\psi(t)=t^2/2$. 

Examining the proof of the density estimates in Proposition \ref{prop:densityest}, to have a uniform scale at which the estimates hold it would be sufficient to have an inequality bounding the the integral of $\kappa_D$ on small sets by a quantity strictly less than their perimeter. In particular, since connected (indecomposable) components of $D$ inherit the curvature of the whole set, such an inequality implies that no arbitrarily small components can be present, a property which (by the dual stability of Proposition \ref{prop:dualstability}) is also true for the denoised level sets with an adequate parameter choice. We formulate this property as the following assumption:
\begin{assumption}
\label{ass:curvineq}
There is a constant $0 < \xi_D < 1$ and a scale $r_0>0$ such that for any $A \subset \R^2$ admitting a variational mean curvature in $L^2(\R^2)$, $x\in \R^2$ and $0< r\ls r_0$ the following inequalities hold:
\begin{equation}\label{eq:curvper}\begin{aligned}\int_{A \cap B(x,r)} \kappa_D^{\,+} &\ls \xi_D \per (B(x,r) \cap A)\text{, and }\\\int_{A \cap B(x,r)} \kappa_D^{\,-} &\ls \xi_D \per (B(x,r) \cap A),\end{aligned}\end{equation}
where $\kappa_D^{\,+} =\max(\kappa_D, 0)$ and $\kappa_D^{\,-} =-\min(\kappa_D, 0)$.
\end{assumption}

Let us check that Assumption \ref{ass:curvineq} holds for the square.

\begin{example}
Denote the unit square by $S \subset \R^2$ and the test set directly by $E:=A \cap B(x,r)$, since we will not use its form or regularity explicitly. By definition $\kappa_S \gs 0$ in $S$, and since $S$ is convex, Proposition \ref{prop:emptyFlambda} implies that $\kappa_S(x) = -\lambda_g g(x)$ for $x \in \R^2 \setminus S$. Let us start with the first inequality of \eqref{eq:curvper}. It is enough to prove that there is $\xi_S < 1$
\[\int_E \kappa_S \ls \xi_S \per(E),\]
for all $E \subset S$ with $|E|$ small and $\operatorname{diam}(E) < 1/2$. Assuming $|E| < |S \setminus C_S| / 4$ with $C_S$ the Cheeger set of $S$, and in view of the optimal curvature of the square \eqref{eq:curvsquare} we have that
\[\int_E \kappa_S \ls \int_{(S \setminus S^\Lambda) \cap Q^1} \kappa_S,\]
for $S^\Lambda$ again in the sense of Definition \ref{def:dlambda}, with some $\Lambda$ such that $|(S \setminus S^\Lambda) \cap Q^1|=|E|$ and $Q^1$ the lower left quadrant. Now, for each $\lambda > 0$
\[|(S \setminus S^\lambda) \cap Q^1|=\frac{1}{\lambda^2}-\frac{\pi}{4 \lambda^2}=\frac{4-\pi}{4\lambda^2}\]
Therefore
\begin{equation}\begin{aligned}\label{eq:intcurvsquarecorner}
\int_{(S \setminus S^\Lambda) \cap Q^1} \kappa_S &= \Lambda|\{\kappa_S > \Lambda\}|+ \int_{\Lambda}^{\infty} |\{\kappa_S > \lambda\}| \dd \lambda \\&= \Lambda |(S \setminus S^\Lambda) \cap Q^1| + \int_{\Lambda}^{\infty} |(S \setminus S^\lambda) \cap Q^1| \dd \lambda \\
&=\frac{4-\pi}{2\Lambda}
\end{aligned}
\end{equation}
and on the other hand, by the isoperimetric inequality
\begin{equation}\label{eq:isopercornerest}\per(E) \gs 2 \sqrt{\pi|E|} = 2 \sqrt{\pi \frac{4-\pi}{4 \Lambda^2}}=\frac{\sqrt{\pi(4-\pi)}}{\Lambda},\end{equation}
and we get the first inequality of \eqref{eq:curvper} with any $\xi_S$ such that
\[1 > \xi_S > \frac{4-\pi}{2 \sqrt{\pi(4-\pi)}}\approx 0.26.\]
To prove the second part of \eqref{eq:curvper} we notice that whenever
\begin{equation}\label{eq:rboundlambdag}r \ls \frac{2\sqrt{\pi}}{\xi_S \,\lambda_g \,|B(0,1)|^{1/2}}\end{equation}
for any $F$ with $\operatorname{diam}(F) \ls r$ we can write
\begin{equation*}\begin{aligned}\frac{1}{\per(F)} \int_F \kappa_S^{\,-} &= \frac{\lambda_g}{\per(F)} \int_{F \setminus S} g \ls \frac{\lambda_g}{\per(F)} \int_F g \\&\ls \frac{\lambda_g |F|}{\per(F)} \ls \frac{\lambda_g |F|}{2\sqrt{\pi} |F|^{1/2}}=\frac{\lambda_g |F|^{1/2}}{2\sqrt{\pi}}\ls \frac{\lambda_g \,r\, |B(0,1)|^{1/2}}{2\sqrt{\pi}} \ls \xi_S,\end{aligned}\end{equation*}
where we have used $g \ls 1$, the isoperimetric inequality and \eqref{eq:rboundlambdag}. 
\end{example}
\begin{remark}
For a convex polygon $P$, one could try to repeat the proof above around a vertex with angle $2\theta$, and $\lambda>0$ large enough so that the contact points of a circle of radius $1/\lambda$ lie on the two edges the vertex belongs to. The analogous formulas to \eqref{eq:intcurvsquarecorner} and \eqref{eq:isopercornerest} are then
\[\int_E \kappa_P \ls \frac{2}{\Lambda} \left(\frac{1}{\tan \theta} + \theta - \frac{\pi}{2} \right) \text{ and } \per(E) \gs 2 \sqrt{\pi|E|} = \frac{2 \sqrt{\pi}}{\Lambda} \sqrt{\left(\frac{1}{\tan \theta}  + \theta - \frac{\pi}{2} \right)}.\]
However, the quotient of these two quantities is below $1$ only for $\theta$ larger than $\approx 0.219$. It is very likely this is a problem of the proof method, since the isoperimetric estimate used is far from sharp, and that in fact Assumption \ref{ass:curvineq} holds for any polygon $P$. Convexity is likely also not required, since a polygon cannot have arbitrarily thin necks, so by the results cited in Remark \ref{rem:offsets} inclusions of balls characterize $P^\lambda$ for $\lambda$ large enough, and we can also determine $P^{-\lambda}$ analogously.
\end{remark}

Now we check that Assumption \ref{ass:curvineq} indeed implies uniform density estimates at a fixed scale. Our scheme will be to work first with the solutions corresponding to noiseless data, and then comparing them with the noisy ones using Proposition \ref{prop:curvgodown}.

\begin{theorem}
\label{thm:cvnoisy}
Let $w \in L^2(\R^2)$ and $\alpha$ satisfying the parameter choice
\begin{equation}\label{eq:superlinearparam}\frac{\|w\|_{L^2}}{\alpha}\ls \eta < 2\sqrt{\pi}\left( 1-\xi_D \right),\end{equation}
and let $u_{\alpha,w}$ denote the corresponding minimizers of \eqref{eq:ROFpsi} with $f=1_D$, where $D$ satisfies Assumption \ref{ass:curvineq} with constant $\xi_D$ and scale $r_0$. Then there is $C_0 \in (0,1)$ such that for a.e.~$s$, the level sets \[E_{\alpha,w}^s:=\{u_{\alpha, w}>s\}\text{ if }s>0 \text{ and }E_{\alpha,w}^s:=\{u_{\alpha, w}<s\}\text{ if }s<0\] satisfy uniform density estimates at scale $r_0$  and with constant $C_0$, that is
\[\frac{|E^s_{\alpha, w} \cap B(x,r)|}{|B(x,r)|} \gs C_0, \text{ and }\frac{|B(x,r) \setminus E^s_{\alpha, w}|}{|B(x,r)|} \gs C_0\]
for all $x \in \partial E^s_{\alpha, w}$ and $0 < r \ls r_0$.
\end{theorem}
\begin{proof}
Since we aim to prove both inner and outer density estimates, we can assume without loss of generality that $s >0$, so that by Proposition \ref{prop:slicedROFpsi} $E^s_{\alpha, w}$ admits the variational curvature $v_{\alpha, w}= 1_D + w - u_{\alpha, w}$. For such $x$ and $r$ we have by the Cauchy-Schwartz inequality, Proposition \ref{prop:curvgodown} and since $\sign(v_{\alpha,0}) = \sign(\kappa_D)$ that
\begin{align}
 \int_{E_{\alpha,w}^s \cap B(x,r)} v_{\alpha,w} & \ls |E_{\alpha,w}^s \cap B(x,r)|^{1/2} \Vert v_{\alpha,w} - v_{\alpha,0} \Vert_{L^2(\R^2)} + \int_{E_{\alpha,w}^s \cap B(x,r)} v_{\alpha,0}\nonumber \\
 & \ls |E_{\alpha,w}^s \cap B(x,r)|^{1/2} \Vert v_{\alpha,w} - v_{\alpha,0} \Vert_{L^2(\R^2)} + \int_{E_{\alpha,w}^s \cap B(x,r)} \kappa_D^{\,+} \nonumber \\
 & \ls |E_{\alpha,w}^s \cap B(x,r)|^{1/2} \Vert v_{\alpha,w} - v_{\alpha,0} \Vert_{L^2(\R^2)} + \xi_D \per(E_{\alpha,w}^s \cap B(x,r)) \nonumber \\
 & \ls |E_{\alpha,w}^s \cap B(x,r)|^{1/2} \eta + \xi_D \per(E_{\alpha,w}^s \cap B(x,r)) \label{eq:bigineq} 
\end{align}
where we have used the first inequality in \eqref{eq:curvper} for the penultimate step and for the last step the parameter choice \eqref{eq:superlinearparam} combined with the dual stability of Proposition \ref{prop:dualstability}, since in this case $\sigma_{\psi^\ast} = \mathrm{Id}$. Plugging this in formula  \eqref{eq:perimcut} of Appendix \ref{sec:densityest}, we obtain
\[ \per(E_{\alpha,w}^s \cap B(x,r)) \left( 1-\xi_D \right) - |E_{\alpha,w}^s \cap B(x,r)|^{1/2} \eta \ls 2 \per\left(B(x,r) ; (E_{\alpha,w}^s)^{(1)}\right).\]
Using now the isoperimetric inequality, we get
\[  |E_{\alpha,w}^s \cap B(x,r)|^{1/2} \left( 2\sqrt{\pi}\left( 1-\xi_D \right) - \eta \right) \ls 2 \per\left(B(x,r) ; (E_{\alpha,w}^s)^{(1)}\right).\]
We can derive the inner density estimate $|E^s_{\alpha, w} \cap B(x,r)| \gs C_0|B(x,r)|$ with $C_0$ depending on $\eta$ by integrating this differential inequality up to $r_0$.

For the outer density, one proceeds in an analogous fashion with the complements $\R^2 \setminus E_{\alpha,w}^s$, which switches the sign of the curvature to $-v_{\alpha, w}$ and makes $\kappa_D^{\,-}$ play a role in \eqref{eq:bigineq} through the second inequality in \eqref{eq:curvper}.
\end{proof}

We notice that in the situation of Theorem \ref{thm:cvnoisy}, the convergence $\dh(\partial E_{\alpha,w}^s, \partial D) \to 0$ for $0<s<1$ follows then by Proposition \ref{prop:compsupp} and Theorem \ref{thm:bdyhausdorff}. Moreover:

\begin{corollary}
With $\alpha_n, w_n$ and $D$ satisfying the assumptions of Theorem \ref{thm:cvnoisy} and for either $s>1$ or $s<0$ we additionally have
\[\limsup_{n \to \infty} \partial E^s_{\alpha_n,w_n} = \emptyset,\]
where $\limsup \partial E^s_{\alpha_n,w_n}$ is defined to be \cite[Def.~4.1]{RocWet98} the set of all limits of subsequences of points in $\partial E^s_{\alpha_n,w_n}$.
\end{corollary}
\begin{proof}
If $s <0$, or $s>1$, by the convergence $u_{\alpha_n,w_n} \to 1_D$ in $L^q$ we have $|E^s_{\alpha_n, w_n}| \to 0$. Assume for a contradiction that we had $x \in \limsup \partial E^s_{\alpha_n,w_n}$. Then we have a not relabelled subsequence and $x_{\alpha_n} \in \partial E^s_{\alpha_n,w_n}$ such that $x_{\alpha_n} \to x$. Now, as in the proof of Theorem \ref{thm:localhausdorff}, by using the inner density estimate, which now holds with constant $C_0$ and for scales $r \ls r_0$ uniformly both in $\alpha$ and the chosen points we get
\[|E^s_{\alpha_n,w_n}| \gs |B(x_{\alpha_n}, r_0) \cap E^s_{\alpha_n,w_n}| \gs C |B(0,r_0)|,\]
which contradicts $|E^s_{\alpha_n, w_n}| \to 0$.
\end{proof}

Observe that in the setting of Theorem \ref{thm:localhausdorff} where the density estimates depend on the distance to $\partial D$, the proof we have given for this corollary fails. Indeed, with such density estimates we could only get that $\dist(x, \partial D) \ls r$ for all $r > 0$ small enough and $x \in \limsup_\alpha \partial E^s_{\alpha,w}$, or
\[\limsup_\alpha \partial E^s_{\alpha,w} \subset \partial D,\]
which for $s<0$ or $s>1$ is not a satisfactory conclusion. We conclude with some further observations about when inequality \eqref{eq:curvper} could be expected to hold.

\begin{remark}
Although it is naturally of $L^2$ scaling for $\kappa_D$, Assumption \ref{ass:curvineq} can be formulated in more general spaces with this scaling, giving some hope that it could hold for sets with Lipschitz boundary. For example, we would have \eqref{eq:curvper} with $\xi_D < 1$ if we had $\Vert \kappa_D \Vert_{L^{2,w}} < 2\sqrt{\pi}$ for the weak $L^2$ norm. In fact, in the notation of \cite[Def. 3.3]{LiTor19}, it is also enough to have $\|\kappa_D\|_{S(\R^2)}<1$, and in \cite[Thm. 3.7]{LiTor19} it is shown that $S(\R^2)$ in fact coincides with the Morrey space $L^{1,1}$ (with different norms, a priori). The quantitative bounds are necessary, since the example in \cite[Ex. 8.4]{LiTor19} provides a set $D$ without density estimates, whose curvature $\kappa_D$ belongs to $L^{1,1}$.
\end{remark}

\begin{remark}\label{rem:cheegervsineq}
We have by definition that
\[\|\kappa_D\|_{L^{2,w}}=\sup_\lambda \lambda\, \big|\{|\kappa_D| \gs \lambda\}\big|^{1/2}.\]
Now, if $D$ is convex the construction of $\kappa_D$ implies that $\kappa_D \gs h(D)$ in $D$, for 
\[h(D)=\inf_{A \subset D} \frac{\per(A)}{|A|}\]
the Cheeger constant of $D$, attained by the unique Cheeger set $C_D$. So with the isoperimetric inequality and that $C_D \subset D$ we have
\[h(D)\, |\{\kappa_D \gs h(D)\}|^{1/2} = h(D)|D|^{1/2}=\frac{\per(C_D)}{|C_D|}|D|^{1/2} \gs 2 \sqrt{\pi} \,\frac{|D|^{1/2}}{|C_D|^{1/2}}\gs 2 \sqrt{\pi},\]
with equality if and only if $D$ is a circle. This means that to use the language of weak norms, it would be necessary to restrict/truncate to small scales or large curvatures.
\end{remark}

\begin{remark}
If $D$ is convex we have that $u_{\alpha,0} = (1-\alpha \kappa_D)^+ 1_D$ (see \cite[Prop.~2.2]{AltCas09} or \cite[Thm.~6]{CasChaNov15}). This implies that one can construct a vector field $z \in L^\infty(\R^d)$ with $|z|\ls 1$ and divergence $\kappa_D$, and which coincides with the normal to $D$ on $\partial D$. The Green formula would provide us with \eqref{eq:curvper} if $z$ was for example continuous in $\accentset{\circ}{D}$, since then cancellations of the flux would appear.
\end{remark}

\begin{remark}
An inequality resembling \eqref{eq:curvper} in Assumption \ref{ass:curvineq} also appears in some works dealing with prescribed mean curvature surfaces in periodic media, like \cite{ChaTho09} and \cite{GolNov12}. In that case, the setting is that of a bounded cell $Q$ and a potential $\tilde{g} \in L^d(Q)$ satisfying $\int_E \tilde{g} \ls (1-\delta)\per(E;Q)$ for fixed $\delta \in (0,1)$ and all $E\subset Q$ is used. In fact, it is proved in \cite[Prop.~4.1]{ChaTho09} using the results of \cite{BouBre03} that in this case there is a continuous vector field $z \in C(Q;\R^d)$ with $|z| \ls 1$ for which $\div z = \tilde{g}$, which is also incompatible with $\tilde{g}$ being the variational mean curvature of a nonsmooth set $D$, since in that case we would expect that $\restr{z}{\partial D} = \nu_D$ \cite[Thm.~3.7]{ChaGolNov15}. This, after Remark \ref{rem:cheegervsineq}, is yet more evidence that \eqref{eq:curvper} can only be expected for small $r$.
\end{remark}

\appendix
\section{Dual problem and its stability}\label{sec:dual}

\begin{proposition}\label{prop:dualROFpsi}
 Assume that $f,w \in L^{d/(d-1)}(\R^d)$. The Fenchel dual of \eqref{eq:ROFpsi} reads
 \begin{equation}
  \label{eq:dualrofpsi} \sup_{v \in \partial \TV(0)} \int v\,(f+w) - \frac 1 \alpha \int \psi^\ast(-\alpha v),
 \end{equation}
 which has a unique maximizer $v_{\alpha,w}$ that satisfies the optimality condition
 \begin{equation}
  \label{eq:vawformula}v_{\alpha,w}=-\frac{1}{\alpha} \psi'(u_{\alpha,w}-f-w)\in \partial \TV(u_{\alpha,w}),
 \end{equation}
where $u_{\alpha,w}$ is the unique minimizer of \eqref{eq:ROFpsi}.
\end{proposition}
\begin{proof}
Existence follows strong duality in Banach spaces \cite[Thm.\ 4.4.3, p.\ 136]{BorZhu05} applied to the space $L^{d/(d-1)}(\R^d)$ with functions $\TV(\cdot)$, $G(\cdot)=\frac{1}{\alpha}\int_{\R^d} \psi(\cdot\,-f-w)$ and the identity operator, while uniqueness is a consequence of strict convexity of $\psi^\ast$.

To apply strong duality we need a qualification condition. Since $G$ is up to a shift and a constant factor the functional $\int \psi$ and $f,w \in L^{d/(d-1)}$, it is enough to check that $u \mapsto \int \psi(u)$ is continuous on $L^{d/(d-1)}$, so that $G$ is in particular continuous at $0$. By \cite[Prop.~IV.1.1]{EkeTem99} continuity holds as soon as we can guarantee that $\psi \circ u \in L^1$ for every $u\in L^{d/(d-1)}$, which is directly implied by the inequality $|\psi(t)|\ls C|t|^{d/(d-1)}$ included in Assumption \eqref{eq:psiassum}. 

The Fenchel conjugate of $G$ reads
\[G^\ast(v) =-\int v\,(f+w) + \frac 1 \alpha \int \psi^\ast(\alpha v).\]
As already computed in \cite[Thm.~1]{IglMerSch18}, the conjugate of the total variation is $\TV^\ast = \rchi_{\partial \TV(0)}$, the indicator function of the convex set $\partial \TV(0)$. In this duality setting, we have \cite[Eqs.~I.(4.24), I.(4.25)]{EkeTem99} the optimality conditions $v_{\alpha,w} \in \partial \TV(u_{\alpha,w})$ and $-v_{\alpha,w} \in \partial G(u_{\alpha,w})$ as well, which are exactly \eqref{eq:vawformula}.
\end{proof}

Now we use assumption \eqref{eq:psiassum} to arrive at a stability result for the maximizers $v_{\alpha,w}$ of \eqref{eq:dualrofpsi}.

\begin{proposition}\label{prop:dualstability}
We have the stability estimate
 \begin{equation}
  \label{eq:dualstability}\|v_{\alpha, w} - v_{\alpha, 0}\|_{L^d(\R^d)} \ls \sigma_\psi \left( \frac{\|w\|_{L^{d/(d-1)}}}{\alpha}\right),
 \end{equation}
where $\sigma_\psi$ is the inverse of the function $t \mapsto m_{\psi^\ast}(t)/t$, with $m_{\psi^\ast}$ the largest modulus of uniform convexity for $\psi^\ast$.
\end{proposition}
\begin{proof}
The computations are analogous to the ones in \cite[Prop.~3.5, Prop.~3.6]{IglMer20}, in turn originating from the methods in \cite{Alb96, AlbNot95}, adapted to the slightly different framework here. The main idea is, for the weak-* closed and convex set $K:=\partial \TV(0) \subset L^d$, to define a generalized projection $\pi:L^{d/(d-1)}\to K$ by
\begin{equation}\label{eq:piK}\pi(u):=\argmin_{v \in K} \int \psi(u) - vu + \psi^\ast(v),\end{equation}
which is single valued by strict convexity of $\psi^\ast$ and then noticing that the dual variable is obtained as
\begin{equation}\label{eq:visproj}v_{\alpha,w} = \frac{1}{\alpha}\pi\left(f+w\right),\end{equation} where we have used that $\psi$ being even implies that $\psi^\ast$ is also even.

Now, given any $u \in L^{d/(d-1)}$ and $v \in L^d$, differentiating the argument of the right hand side of \eqref{eq:piK} in direction $\pi(u)-v$ and using minimality at $\pi(u)$ we end up with
\begin{equation}\label{eq:piKineq}\int \left(v - \pi(u)\right)\left(u - (\psi^\ast)' \circ \pi(u)\right) \gs 0.\end{equation}
Moreover, we have the uniform monotonicity inequality (for a proof, see \cite[Lem.~1.2]{IglMer20})
\begin{equation}\label{eq:unifmonoton}\begin{gathered}\int (\pi(u_1) - \pi(u_2))\big((\psi^\ast)' \circ \pi(u_1) - (\psi^\ast)' \circ \pi(u_2)\big) \\ \gs 2 m_{\psi^\ast}\big(\|\pi(u_1)-\pi(u_2)\|_{L^{d/(d-1)}}\big),\end{gathered}\end{equation}
for whose left hand side we have, using \eqref{eq:piKineq} twice (with $v=\pi(u_1)$ and $v=\pi(u_2)$) and H\"older inequality, that
\begin{equation}\label{eq:csproj}\begin{aligned}
&\int \left(\pi(u_1) - \pi(u_2)\right)\left((\psi^\ast)' \circ \pi(u_1) - (\psi^\ast)' \circ \pi(u_2)\right) \\ 
&\quad\ls \int \left(\pi(u_1) - \pi(u_2)\right)\left(u_1-u_2\right) \\
&\qquad+ \int \left(\pi(u_1) - \pi(u_2)\right)\left((\psi^\ast)' \circ \pi(u_1) - u_1\right) \\
&\qquad- \int \left(\pi(u_1) - \pi(u_2)\right)\left((\psi^\ast)' \circ \pi(u_2) - u_2\right) \\
&\quad\ls \int \left(\pi(u_1) - \pi(u_2)\right)\left(u_1-u_2\right) \\
&\quad\ls \left\|\pi(u_1) - \pi(u_2)\right\|_{L^d} \left\|u_1 - u_2\right\|_{L^{d/(d-1)}}.
\end{aligned}\end{equation}
The combination of \eqref{eq:csproj}, \eqref{eq:unifmonoton} and \eqref{eq:visproj} allows us then to conclude \eqref{eq:dualstability}. As already noted in \cite{IglMer20}, the property (see \cite[Fact 5.3.16]{BorVan10}) $m_{\psi^\ast}(ct)>c^2m_{\psi^\ast}(t)$ for all $c>1$ implies that the function $t \mapsto m_{\psi^\ast}(t)/t$ is strictly increasing, so its inverse is well defined.
\end{proof}

\section{Density estimates for denoised level sets}\label{sec:densityest}
\begin{proposition}\label{prop:densityest}
Let $K \subset \R^d$ be a bounded set, assume that
\begin{equation}\label{eq:isoperstabil2}\|v_{\alpha, w} - v_{\alpha, 0}\|_{L^d(\R^d)} \ls C_0 < \Theta_d,\end{equation}
which is possible by Proposition \ref{prop:dualstability}. Furthermore, assume that for each $\eps >0$ there is $r_{K,\eps} >0$ such that for all $x \in \R^d \setminus K$ and all $\alpha$ we have the equi-integrability estimate
\begin{equation}\label{eq:equiint}\int_{B(x,r_{K,\eps})} |v_{\alpha, 0}|^d \ls \eps^d.\end{equation}
Then the level sets $E^s_{\alpha, w}$ denoting $\{u_{\alpha, w} > s\}$ for $s>0$ and $\{u_{\alpha, w} < s\}$ when $s<0$ satisfy uniform density estimates at some scale $r_K$ and with constant $C_K$ outside $K$, that is
\[\frac{|E^s_{\alpha, w} \cap B(x,r)|}{|B(x,r)|} \gs C_K, \text{ and }\frac{|B(x,r) \setminus E^s_{\alpha, w}|}{|B(x,r)|} \gs C_K\]
for all $x \in \partial E^s_{\alpha, w} \setminus K$ and $0 < r \ls r_K$.
\end{proposition}
\begin{proof}
Let $x \in \partial E^s_{\alpha, w} \setminus K$. We start from the formula
\begin{equation}\label{eq:perimcut}
 \per(E^s_{\alpha, w} \cap B(x,r)) - \int_{E^s_{\alpha, w} \cap B(x,r)} v_{\alpha,w} \ls 2 \per\left(B(x,r) ; (E^s_{\alpha, w})^{(1)}\right),
\end{equation}
which holds for almost every $r>0$ \cite[Lem.~8]{IglMerSch18} by repeated application of the precise formulas for the perimeter of an intersection \cite[Thm.~16.3]{Mag12}.

On the other hand we have, thanks to the H\"older inequality, the condition \eqref{eq:isoperstabil2} and local equiintegrability \eqref{eq:equiint} that for $0 < r \ls r_{K,\eps}$
\begin{align*}
 \int_{E^s_{\alpha, w} \cap B(x,r)} v_{\alpha,w} & \ls |E^s_{\alpha, w} \cap B(x,r)|^{(d-1)/d} \Vert v_{\alpha,w} - v_{\alpha,0} \Vert_{L^d(\R^d)} + \int_{E^s_{\alpha, w} \cap B(x,r)} |v_{\alpha,0}| \\
 & \ls |E^s_{\alpha, w} \cap B(x,r)|^{(d-1)/d} \left(C_0 + \eps\right) 
\end{align*}
Plugging this in \eqref{eq:perimcut}, we obtain (the superscript $(1)$ denoting density-$1$ points)
\[ \per(E_{\alpha,w}^s \cap B(x,r)) - |E_{\alpha,w}^s \cap B(x,r)|^{(d-1)/d} (C_0+\eps) \ls 2 \per\left(B(x,r) ; (E^s_{\alpha, w})^{(1)}\right).\]
Using now the isoperimetric inequality, we get
\begin{equation}\label{eq:diffineq}|E_{\alpha,w}^s \cap B(x,r)|^{(d-1)/d} \left( \Theta_d-C_0-\eps \right) \ls 2 \per\left(B(x,r) ; (E^s_{\alpha, w})^{(1)}\right).\end{equation}
Taking some fixed $\eps_0 < \Theta_d - C_0$, and since for a.e. $r>0$
\[\per\left(B(x,r) ; (E^s_{\alpha, w})^{(1)}\right) = \mathcal{H}^{d-1}\left( \partial B(x,r) \cap (E^s_{\alpha, w})^{(1)}\right) = \frac{d}{dt}\bigg\vert_{t=r}\left|E^s_{\alpha, w}\cap B(x,t)\right|,\]
we can derive the inner density estimate $|E^s_{\alpha, w} \cap B(x,r)| \gs C_K |B(x,r)|$ with $C_K$ depending on $\eps_0$ by integrating the differential inequality \eqref{eq:diffineq} up to $r_K:=r_{K,\eps_0}$. 

The outer density estimate follows analogously by considering the complement $\R^d \setminus E^s_{\alpha, w}$, which admits the variational mean curvature $-v_{\alpha,w}$.
\end{proof}

\bibliographystyle{plain}
\bibliography{conv_without_source_cond}

\end{document}